\documentclass[10pt]{article}
\usepackage[margin=1in,vmargin=1in]{geometry}
\usepackage{tikz}
\usepackage{selectp}
\usepackage{graphicx}
\usepackage{url}
\usepackage{multirow}
\usepackage{setspace}
\usepackage{amssymb}
\usepackage{comment}
\usepackage{amsmath}
\usepackage{amsthm}
\newtheorem{example}{Example}
\newtheorem{remark}{Remark}
\newtheorem{lemma}{Lemma}
\newtheorem{proposition}{Proposition}

\newcommand\Prob{{\mathbb {P}}}
\newcommand\Q{{\bf Q}}

\newcommand\convLaw{\buildrel \mathcal{D} \over \longrightarrow}
\newcommand\approxLaw{\buildrel \mathcal{D} \over \approx}
\newcommand\vecX{{\bf X}}
\newcommand\field{\mathcal{F}}
\newcommand\alphabet{\mathcal{A}}
\newcommand\Gt{\widetilde{G}}
\newcommand{\Phic}{\Phi_{\collection}}
\newcommand{\Phitc}{\widetilde{\Phi}_{\collection}}
\newcommand{\collection}{\mathcal{C}}
\newcommand\normal{\mathcal{N}}
\newcommand\bfzero{{\bf 0}}
\newcommand\bfmu{\mbox{\boldmath $\mathbf\mu$}}
\newcommand\bfSigma{\mbox{\boldmath $\mathbf\Sigma$}}
\newcommand\indexingset{\mathcal{I}}
\newcommand\given{{\, | \, }}
\newcommand\E{{\mathbb {E}}}
\newcommand\V{{\mathbb {V}{\rm ar}}}
\newcommand\Cov{{\mathbb {C}{\rm ov}}}

\newtheorem{theorem}{Theorem}

\newcommand{\Nz}{N_{z}}
\newcommand{\fs}{f^{*}}

\newcommand{\A}{\mathcal{A}}

\newcommand{\vos}{v_{1}^{*}}
\newcommand{\vts}{v_{2}^{*}}
\newcommand{\vhs}{v_{3}^{*}}
\newcommand{\Pw}{\Prob(w)}
\newcommand{\Iwi}{I_{N_{z},T_\nu,w}}
\newcommand{\Iwj}{I_{N_{z},T_\kappa,w}}

\newcommand{\Astar}{\A^{*}}

\newcommand{\sumw}{\sum_{w \in \Astar}}
\begin{document}
\begin{center}
{\large \bf  On the variety of shapes 
in digital trees 

\bigskip
Jeffrey Gaither\footnote{Mathematical Biosciences Institute, The Ohio State University,  Jennings Hall 3rd Floor, 1735 Neil Ave.  Columbus, OH~43210 U.S.A.,
  Email: \url{gaither.16@mbi.osu.edu}}
\quad Hosam Mahmoud\footnote{Department of Statistics, The George Washington
  University, 801 22nd Street, Washington, D.C.~20052 U.S.A., Email: \url{hosam@gwu.edu}}
\quad
Mark Daniel Ward\footnote{Department of Statistics, Purdue University, 150 North
  University Street, West~Lafayette, IN~47907--2067 U.S.A.,
  Email: \url{mdw@purdue.edu}}
\bigskip
}

\today
\end{center}

\begin{abstract}
We study the joint distribution of the number of occurrences of members of a collection of nonoverlapping
motifs in digital data. We deal with finite and countably infinite collections.
For infinite collections, 
the setting requires that we be very explicit about the specification
of the underlying measure-theoretic formulation.
We show that (under appropriate normalization) for such a collection,
any linear combination of 
the number of occurrences of each of the motifs in the data
has a limiting normal distribution. In many instances,
this can be interpreted in terms of the number of occurrences
of individual motifs: They have a multivariate normal distribution. 
The methods of proof include combinatorics on words, integral transforms, 
and poissonization.
\end{abstract}

                       
\section{Introduction}
\label{Sec:trie}
With all types of data and their supporting storage one is often interested in substructures.
In a text we are interested in the occurrence of certain words, such as cancerous
genes in DNA strands. 
When digital data are stored in digital
trees we wish to identify the occurrence of certain tree shapes, which
we call \emph{motifs}.
Certain motifs may indicate particular properties of the digital records
stored, such as the prevalence of a certain disease in DNA data. Often,
 the presence of a particular
substructure is significant in the presence of certain other structures, 
such as 
the alleles of cancer, which become more serious in the presence of
certain other alleles. 
So, we are interested in the \emph{joint occurrence} of members of a
collection of shapes in a given tree.   
There can also be applications
in data compression. When a certain small tree shape 
occurs multiple times in a large tree,
we can store the data in these smaller trees using a simpler format,
with only one pointer in each structure to their common tree shape. 
This allows us to store only one actual copy of each subtree shape.
 
We consider 
$m$-ary tries, which are trees arising from random
strings over an $m$-ary alphabet. 
The trie was introduced in \cite{Briandais,Fredkin} for information
re\emph{trie}val.
In addition to their use as data structures, 
tries support the operation of---and serve as models for---the analysis 
of several important algorithms, such as 
Radix Exchange Sort \cite{Knuth}, 
and Extendible Hashing \cite{Fagin}.
 
We assume that our digital data are 
infinite strings written using the symbols of an $m$-ary alphabet
$$\alphabet = \{a_1, \ldots, a_m\}.$$ 
In the sequel $\alphabet^*$ will denote the 
set of all finite-length words using 
letters from $\alphabet$.
Each string is generated independently of all others 
by a probabilistic memoryless
source, i.e., the successive symbols of one 
string are generated independently,
and the probability of the source emitting 
the symbol $a_j \in \alphabet$ is $\Prob(a_j) = p_j$.
To avoid trivialities, we assume~$p_j > 0$ 
for $j = 1, \ldots, m$. 

Tries are a form of digital tree. 
They have a recursive definition. 
An $m$-ary trie on $n$ strings is empty, when $n=0$. 
Nonempty tries on $n \ge 1$ strings have
two types of nodes: internal (which serve the purpose of branching) 
and external (each of which 
contains one string). Each internal node has $m$ 
subtrees (some may be empty), corresponding to the symbols $a_1,
\ldots, a_m$ (respectively, from left to right).
An $m$-ary trie
on $n = 1$ strings holds one string;  the trie 
consists of an external node carrying that string.
An $m$-ary trie
on $n > 1$ strings consists of a root node of the internal type, and 
$m$ subtrees, which are themselves $m$-ary tries. All the strings
starting with $a_j$ go into the $j$th subtree. The recursion continues 
in the subtrees, with branching from the $\ell$th to $(\ell+1)$st
level according to the $(\ell+1)$st symbol
in the strings. Henceforth, we shall let the ``$m$'' be implicitly
understood, and often call an $m$-ary trie simply a trie. 
The number of strings in a trie is its \emph{size}.

Figure~\ref{Fig:example} instantiates the definition of 
tries with a quaternary trie, of size 12, constructed from twelve DNA strands, where 
the alphabet is the set of nucleotides $\{{\tt A}, {\tt C}, {\tt G},
{\tt T}\}$. The 12
strings in the external nodes are 
\begin{center}
\begin{minipage}{2.5in}
\begin{align*}
S_{1} &= {\tt CATCTGGTA}\ldots\\
S_{2} &= {\tt AATACTTCG}\ldots\\
S_{3} &= {\tt TGCCGAATC}\ldots\\
S_{4} &= {\tt TTTGTTCTA}\ldots\\
S_{5} &= {\tt AAGATGGAA}\ldots\\
S_{6} &= {\tt GCAAATCTG}\ldots
\end{align*}
\end{minipage}
\begin{minipage}{2.5in}
\begin{align*}
S_{7} &= {\tt GCTCTGGTA}\ldots\\
S_{8} &= {\tt AAACTGGTA}\ldots\\
S_{9} &= {\tt TGGTACCCG}\ldots\\
S_{10} &= {\tt GCATCTGGT}\ldots\\
S_{11} &= {\tt ATGTCTGGT}\ldots\\
S_{12} &= {\tt GCAGTGGTA}\ldots
\end{align*}
\end{minipage}
\end{center}

\begin{figure}[ht]
\begin{center}
\begin{tikzpicture}
  [level distance=16mm,
    level 1/.style={sibling distance=30mm},
    level 2/.style={sibling distance=8mm},
     level 3/.style={sibling distance=8mm},
     level 4/.style={sibling distance=8mm}
    ]
\tikzset{every node/.style={minimum size = 8, inner sep=0}}  
\node[circle,draw,fill = black]{\phantom{.}} 
  child[sibling distance=20mm]{node[circle, draw, fill = black]{\phantom{.}}
    child{node[circle, draw, fill = black]{\phantom{.}}
      child{node[minimum size = 15, inner sep=0, draw]{$S_8$}}
      child{node[minimum size = 15, inner sep=0, draw, dashed]{} edge from parent [dashed]}
      child{node[minimum size = 15, inner sep=0, draw]{$S_5$}}
      child{node[minimum size = 15, inner sep=0, draw]{$S_2$}}
    }
    child{node[minimum size = 15, inner sep=0, draw, dashed]{} edge from parent [dashed]}
    child{node[minimum size = 15, inner sep=0, draw, dashed]{} edge from parent [dashed]}
    child{node[minimum size = 15, inner sep=0, draw]{$S_{11}$}}
  }
  child[sibling distance=15mm]{node[minimum size = 15, inner sep=0, draw]{$S_1$}} 
  child[sibling distance=35mm]{node[circle, draw,  fill = black]{$\phantom{.}$}
    child{node[minimum size = 15, inner sep=0, draw, dashed]{} edge from parent [dashed]} 
    child{node[circle, draw,  fill = black]{$\phantom{.}$}
      child{node[circle, draw,  fill = black]{\phantom{.}}
        child{node[minimum size = 15, inner sep=0, draw]{$S_6$}}
        child{node[minimum size = 15, inner sep=0, draw, dashed]{} edge from parent [dashed]}
        child{node[minimum size = 15, inner sep=0, draw]{$S_{12}$}}
        child{node[minimum size = 15, inner sep=0, draw]{$S_{10}$}}
      }
      child{node[minimum size = 15, inner sep=0, draw, dashed]{} edge from parent [dashed]}
      child{node[minimum size = 15, inner sep=0, draw, dashed]{} edge from parent [dashed]}
      child{node[minimum size = 15, inner sep=0, draw]{$S_7$}}
    } 
    child{node[minimum size = 15, inner sep=0, draw, dashed]{} edge from parent [dashed]} 
    child{node[minimum size = 15, inner sep=0, draw, dashed]{} edge from parent [dashed]} 
  } 
  child[sibling distance=35mm]{node[circle, draw,  fill = black]{\phantom{.}}
    child{node[minimum size = 15, inner sep=0, draw, dashed]{} edge from parent [dashed]}
    child{node[minimum size = 15, inner sep=0, draw, dashed]{} edge from parent [dashed]}
    child{node[circle, draw,  fill = black]{\phantom{.}}
      child{node[minimum size = 15, inner sep=0, draw, dashed]{} edge from parent [dashed]}
      child{node[minimum size = 15, inner sep=0, draw]{$S_3$}}  
      child{node[minimum size = 15, inner sep=0, draw]{$S_9$}} 
      child{node[minimum size = 15, inner sep=0, draw, dashed]{} edge from parent [dashed]}    
    }
    child{node[minimum size = 15, inner sep=0, draw]{$S_4$}}    
  };
\end{tikzpicture} 
\end{center}
  \caption{Example of a quaternary trie of size 12 for DNA data.}
  \label{Fig:example}
\end{figure}
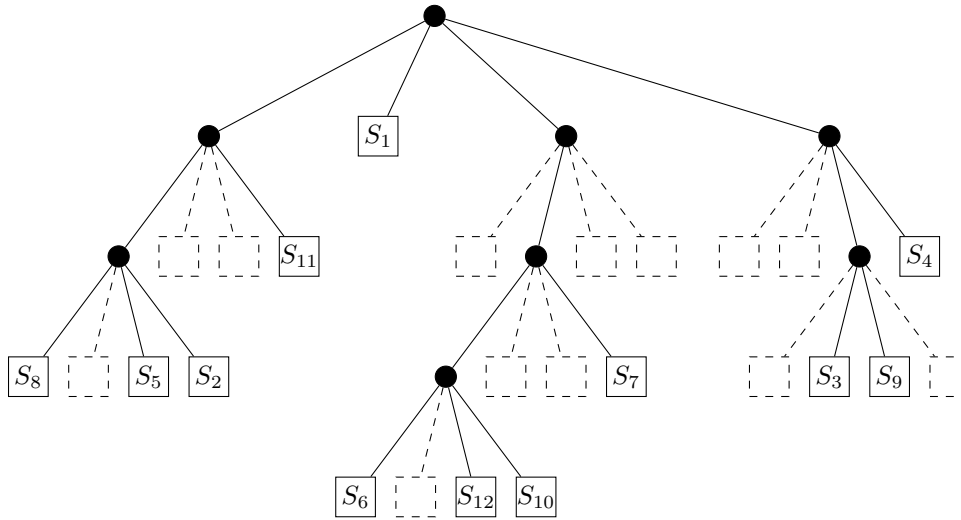

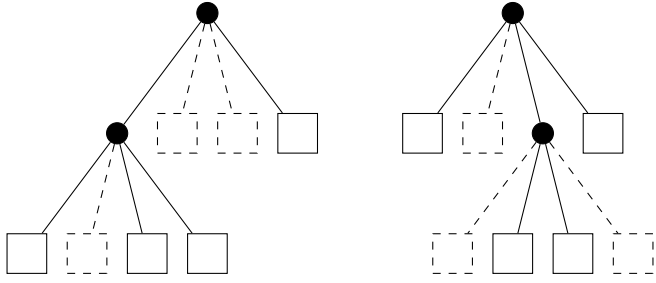
\begin{figure}
\begin{center}
\begin{tikzpicture}
  [level distance=16mm,
    level 1/.style={sibling distance=8mm},
    level 2/.style={sibling distance=8mm},
    ]
\tikzset{every node/.style={minimum size = 8, inner sep=0}}  
\node[circle, draw,  fill = black]{$\phantom{.}$}
      child{node[circle, draw,  fill = black]{\phantom{.}}
        child{node[minimum size = 15, inner sep=0, draw]{}}
        child{node[minimum size = 15, inner sep=0, draw, dashed]{} edge from parent [dashed]}
        child{node[minimum size = 15, inner sep=0, draw]{}}
        child{node[minimum size = 15, inner sep=0, draw]{}}
      }
      child{node[minimum size = 15, inner sep=0, draw, dashed]{} edge from parent [dashed]}
      child{node[minimum size = 15, inner sep=0, draw, dashed]{} edge from parent [dashed]}
      child{node[minimum size = 15, inner sep=0, draw]{}};
\end{tikzpicture} 
\hskip 1cm
\begin{tikzpicture}
  [level distance=16mm,
    level 1/.style={sibling distance=8mm},
    level 2/.style={sibling distance=8mm},
    ]
\tikzset{every node/.style={minimum size = 8, inner sep=0}}  
  \node[circle, draw,  fill = black]{\phantom{.}}
        child{node[minimum size = 15, inner sep=0, draw]{}
          edge from parent}
    child{node[minimum size = 15, inner sep=0, draw, dashed]{} edge from parent [dashed]}
    child{node[circle, draw,  fill = black]{\phantom{.}}
      child{node[minimum size = 15, inner sep=0, draw, dashed]{} edge from parent [dashed]}
      child{node[minimum size = 15, inner sep=0, draw]{}}  
      child{node[minimum size = 15, inner sep=0, draw]{}} 
      child{node[minimum size = 15, inner sep=0, draw, dashed]{} edge from parent [dashed]}    
    }
    child{node[minimum size = 15, inner sep=0, draw]{}};
\end{tikzpicture}
\end{center}
  \caption{Two nonoverlapping DNA motifs of size 4.  The motif on the left
    occurs twice in Figure~\ref{Fig:example}; their roots are the parents of
    $S_7$ and $S_{11}$.  
    The motif on the left corresponds to the collection of strings
    $\{ {\tt AA\ldots}, {\tt AG\ldots}, {\tt AT\ldots}, {\tt T\ldots} \}$.
    The second motif does not occur in
    Figure~\ref{Fig:example}. 
    The second motif corresponds to the collection of strings
    $\{{\tt A\ldots}, {\tt GC\ldots}, {\tt GG\ldots}, {\tt T\ldots} \}$.}
  \label{Fig:motifs}
\end{figure}

The rest of this paper is organized as follows. 
In Section~\ref{Sec:Technical},
we lay out the general setup and the scope of the investigation.
In Section~\ref{Sec:Results}, we present
the main results. In Section~\ref{Sec:ProbabilitySpace},
we give the measure-theory formulation by giving
a probability space on which all the random variables in the paper
are formally defined.
In Section~\ref{Sec:poissonization}, we take up poissonization. 
In Section~\ref{Sec:Proofs}, we present proofs. The proofs are structured in subsections:
Subsection~\ref{Subsec:mean} is for the derivation of the mean; 
Subsection~\ref{Subsec:var} is
for the derivation of the variance, and is followed by several
subsections dealing with technical details: Mellin transform
and some motivating
words about this tool (Subsection~\ref{Mellintransformsection}), its existence domain~(Subsection~\ref{Fundamentalstripsection}),
variance asymptotics~(Subsection~\ref{AsymptoticsofVsection}) and the covariance structure (Subsection~\ref{Subsec:Covar}).
The moment generating function of the univariate 
linear combination is dealt with in Subsection~\ref{Subsec:char},
where a recurrence is given.
In Subsection~\ref{Susbec:limit} we derive a Gaussian limit for the distribution of the combined occurrences of an arbitrary linear combination of motifs,  which we then discuss  in 
examples in Section~\ref{Sec:examples} (writing one subsection for each example).

A similar investigation has been carried out in~\cite{Gopaladesikan}
on recursive trees, but it required a rather different set of probabilistic tools.
 There are many other examples in the literature about 
pattern counting in other random tree structures.  We mention only a
few here.
P.~Flajolet, X.~Gourdon, and C.~Mart{\'\i}nez~\cite{FGM1997} investigated 
subtrees on the fringe of the binary search tree, of a certain size
but not a certain shape.
J.~Fill \cite{Fill} also has studied a distribution on the set of binary search
trees, in the context of a random permutation model.
In addition to identifying patterns in trees, 
a recent paper by Gopaladesikan, Wagner, and
Ward~\cite{ForbiddenMotifs} considers missing patterns in trees.
It would be impractical to give a full survey of the myriad papers
that have results about patterns in random trees.

\section{Technical development}
\label{Sec:Technical}
We assume that a (random) $m$-ary trie is built from $n$ random strings.
For a given motif (trie shape) $T$, let
$X_{n,T}$ count the number of occurrences of~$T$ on the fringe of
a random trie of size $n$. 
By occurrence on the fringe we mean that $T$ coincides in shape with a
maximal rooted subtree of the trie, in the sense that the subtree does
not contain a subtree with more nodes than in $T$. 
 
When the motif $T$ is the trie on the left-hand side of Figure~\ref{Fig:motifs}, 
there are $X_{12, T} = 2$ occurrences of it
in the trie of Figure~\ref{Fig:example}. 
When the motif $\widetilde T$ is the trie on the right-hand side of Figure~\ref{Fig:motifs}, 
there are $X_{12, \widetilde T} = 0$ occurrences of it
in the trie of Figure~\ref{Fig:example}.
The roots of the two
occurrences of $T$ are the parents of $S_7$ and $S_{11}$.
In all figures in this paper,
empty subtrees are shown as dashed external nodes,
connected to their parents via dashed edges.

Let $\indexingset$ be an indexing set, 
of cardinality at most $\aleph_0$.
Let 
$$\collection  = \{T_\nu \given \nu \in \indexingset \}$$
be a given collection of motifs. 
We say that two motifs are \emph{nonoverlapping}, if neither 
appears as a subtree in another, and 
we call a collection of motifs 
a \emph{collection of nonoverlapping motifs}, if its members are
pairwise nonoverlapping.  For instance, neither of the two motifs in
Figure~\ref{Fig:motifs} appears as a subtree of the other, so this is
a collection of two nonoverlapping motifs.

In many applications such
a collection will be finite, but our presentation covers cases of
countably infinite collections, too. 
Countably infinite nonoverlapping collections arise
naturally in many applications, such as the case discussed 
in the following example.

\begin{example}
In~\cite{Ward} the average of the number of ``$\tau$--cousins,"
which are any tries of size~$\tau$ on the fringe of a random trie, was found. 
In the notation of the present paper, 
if $\mathcal{C}$ denotes the collection of all motifs corresponding to
$\tau$-cousins, then the number of $\tau$--cousins in a trie is
$\sum_{T_\nu \in \mathcal{C}} X_{n, T_\nu}$.
In particular, we note that there is a countably infinite number of
$\tau$-cousins.  Thus, we can use an indexing set~$\mathcal{I}$ that is in
one-to-one correspondence with the positive natural numbers~$\mathbb{N}$.
\end{example}

A trie is basically a correspondence between 
a set of strings and a tree structure.
For $n\geq 2$, if $(W_{1},\ldots,W_{n})$ is an ordered $n$-tuple of
words of finite length (i.e., $W_{j}\in\alphabet^{\geq 1}$), we say that 
$(W_{1},\ldots,W_{n})$ has the \emph{trie property} if, for each $i$:
(1) $W_{i}$ is not a prefix of any of the other $W_{j}$'s, and (2) if
the last character of $W_{i}$ is removed, it becomes a prefix of at
least one of the other $W_{j}$'s. 
(In the case $n=1$, this must have simply $W_{1}=\varepsilon$, namely,
the empty word, has the \emph{trie property}.)

A trie with $n$ leaves always uniquely corresponds to a set of $n$
strings with the \emph{trie property}.  For example, the 12-tuple of strings
that induces the trie displayed in Figure~\ref{Fig:example} is:
$$(W_{1},\ldots,W_{12}) = ({\tt C},{\tt AAT},{\tt TGC},{\tt TT},{\tt AAG},{\tt GCAA},{\tt GCT},{\tt AAA},{\tt TGG},{\tt GCAT},{\tt AT},{\tt GCAG}).$$

In our results, we will utilize the data entropy function
$$h = h(p_1, \ldots, p_m) = - \sum_{j=1}^m p_j \ln p_j.$$
Also, we use $\Q(T)$ to denote the probability that a trie grown on
$\tau$ random strings coincides with a given fixed motif $T$ of size
$\tau$. Some authors call 
such a probability a \emph{shape functional}. See~\cite{Dobrow, Fill, Kapur} 
for counterpart definitions in $m$-ary search trees, and~\cite{Feng} 
for the counterpart in recursive trees. 
These two classes of trees require probabilistic tools that are
rather different from the analytic probability tools utilized in this
paper for digital trees.
\begin{remark}
\label{Remark:shapefunctional}
Consider a trie grown from the $\tau$ strings
$$S_j =  a_{j,1} a_{j, 2} \ldots a_{j, L_j} \ldots, \qquad \mbox{for \ } j=1, \ldots, \tau,$$
where $L_j$ denotes the length of the shortest prefix that uniquely
identifies $S_j$ among $S_{1},\dots,S_{\tau}$.  
The same trie shape (motif) $T$ arises, regardless of the $\tau!$
possible orderings of insertion of these $\tau$ strings, so the leaves are labeled with $S_1, \ldots, S_\tau$. The
motif\thinspace~$T$ has shape functional 
$$\Q(T) = \tau! \prod_{j=1}^\tau \ \prod_{s=1}^{L_j}\Prob(a_{j, s}).$$
\end{remark}
\section{Results}
\label{Sec:Results}
The main results, in terms of averages and covariances, are given next.

\bigskip
\begin{proposition}
\label{Prop:mean}
Let $X_{n,T}$ be the number of occurrences 
of a fixed motif\thinspace~$T$ of  size $\tau$
in an $m$-ary trie constructed
over $n$ independent strings from an
$m$-ary alphabet $\{a_1, \ldots, a_m\}$, with probabilities $p_j > 0$, 
for $j=1, \ldots, m$. We then have
$$
\E[X_{n,T}] = {\Q(T) \over \tau(\tau-1) h} \, n
             +   \xi_T (n) n + o(n),
$$
where $\xi_T$ is a possibly fluctuating function 
 with average value zero.  
\end{proposition}
 We note that $\xi_T$ usually has 
small magnitude in many specific cases, when the probabilities $(p_{1},\dots,p_{m})$  
are periodic\footnote{A set of probabilities $p_1, \ldots, p_m$ is said to be periodic,
when $\log p_j / \log p_k$ is rational, 
for every $1 \le j, k\le m$.} (as an example, the magnitude can be of the order
$10^{-5}$ for some specific values of the $p_i$'s), and is $0$, otherwise.
(We do not claim, however, that any uniform small
  bound exists, which covers all $(p_{1},\ldots,p_{m})$.)

\begin{remark}\label{Rem:interpretsum}
The average in Proposition~\ref {Prop:mean} is the 
same as the average number of $\tau$--cousins in~\cite{Ward}, except for the factor
$\Q(T)$. This is, of course, to be expected, as $\tau$--cousins can come in various shapes (all being
tries of  size $\tau$), and the expected number of occurrences of a given shape
is the same as the average number of cousins, ramified by the shape functional, 
which is the probability of picking
the shape in question.
\end{remark}

\bigskip 
\begin{theorem}
\label{Theo:covar}
Let $X_{n,T}$ be the number of occurrences 
of a fixed motif\thinspace~$T$ of size~$\tau$
in an $m$-ary trie constructed
over $n$ independent strings from an
$m$-ary alphabet $\{a_1, \ldots, a_m\}$, with probabilities $p_j > 0$, 
for $j=1, \ldots, m$. Then, we have
\begin{align*}
\V[X_{n, T}] &= \biggl[{\Q(T) \over \tau (\tau-1)h}\, 
- {2 \Q^2(T) \over  h}\bigg(\frac{2^{-2\tau}} {2\tau(2\tau-1) } {2\tau
  \choose \tau} +\frac1  {(\tau!)^2}\sum_{j = 0}^\infty (-1)^j
{\sum_{k=1}^m p_k^{j+\tau} \over 1 - \sum_{k=1}^m p_k^{j+\tau}} \times
{(j+ 2\tau-2)! \over j!}\bigg)\\
&\qquad{}+ \delta_{T}(n) - \Big({\Q(T) \over \tau(\tau-1)h} + \widehat \delta_T(n)\Big)^2\biggr]\, n +o(n).
\end{align*}
where $\Q(T)$ is the shape functional of $T$, and
$\delta_T(.)$ and  $\widehat \delta_T(.)$ are possibly fluctuating 
 with average value zero,
when the probability set is aperiodic, and is $0$
otherwise.\footnote{In the aperiodic case, the $o(n)$ estimate 
can be improved to $O(n^{1-\varepsilon})$, 
for some $0< \varepsilon < 1$.}

Furthermore,
if $T$ and $\widetilde T$ are two nonoverlapping 
shapes of sizes $\tau$ and $\widetilde \tau$ (where $\tau$ and
$\widetilde{\tau}$ are not necessarily the same), we have the covariance
\begin{align*}
\Cov[X_{n,T} , X_{n,\widetilde T}]
&=\bigg[-{2 \Q(T)\, \Q(\widetilde T) \over \tau!\, \widetilde \tau!\,
  h}  \bigg(2^{-\tau -\widetilde\tau}(\tau+\widetilde \tau-2)!\\
&\qquad\qquad\qquad {}+2^{-1}\sum_{j = 0}^\infty (-1)^j\bigg( {\sum_{k=1}^m p_k^{\tau+j} \over 1 - \sum_{k=1}^m p_k^{\tau+j}}+ {\sum_{k=1}^m p_k^{\widetilde\tau+j} \over 1 - \sum_{k=1}^m p_k^{\widetilde\tau+j}}\bigg) \times {(\tau+ \widetilde \tau+j-2)! \over j!}\bigg) \cr
& \qquad\qquad  {} + \frac 1 2\bigl( \delta_{T, \widetilde T}(n) - \delta_T (n) - \delta_{\widetilde T} (n)\bigr)   - {\Q(T)\, \Q(\widetilde T) \over  \tau(\tau-1)\widetilde\tau(\widetilde\tau-1)h^{2}}  \cr
& \qquad \qquad {}+ {1 \over 2h}\bigg({\Q(T) \over \tau(\tau-1)}\bigl(\widehat \delta_{T}(n)-\widehat\delta_{T,\widetilde T}(n)\bigr) + {\Q(\widetilde T) \over \widetilde \tau(\widetilde \tau-1)} \bigl(\widehat\delta_{\widetilde T}(n)-\widehat\delta_{T,\widetilde T}(n)\bigr)\bigg)\cr
&\qquad\qquad {} -(\widehat\delta_{T,\widetilde T}(n)^{2}- \widehat\delta_{T}(n)^{2}-\widehat\delta_{\widetilde T}(n)^{2})
\bigg]\, n+o(n),
\end{align*}
where $\delta_{T}(.), \widehat\delta_{T}(.)$, and $\delta_{T,\widetilde T}(.),
\widehat\delta_{T,\widetilde T}(.)$ are 
oscillating functions (possibly 0),
and the first two are the same as those that appear in the variance.
\end{theorem}


Another main result of this paper is the following theorem and its corollary.
These results use a terminology from multivariate statistics. The notation
$\normal_k(\bfzero ,\bfSigma)$ stands for
the multivariate jointly normally distributed
random vector with mean vector~$\bfzero$ (of $k$ components) 
and $k \times k$ covariance 
matrix $\bfSigma$. When $k=1$, we shall write the univariate
normal variate in the usual form as $\normal(0, \sigma^2)$,
where the 0 and $\sigma^2$ are both scalars.

\bigskip
\begin{theorem}
\label{Theo:main}
Let $\collection = \{T_\nu \given \nu \in \indexingset\}$ be a collection of 
nonoverlapping tries, all of size $\tau > 1$,
where $\indexingset$ is finite or countably infinite.
Let $X_{n,T}$ be the number of
occurrences of a shape~$T$ of  size $\tau$
 in an $m$-ary trie constructed
over $n$ independent strings from an
$m$-ary alphabet $\{a_1, \ldots, a_m\}$, with probabilities 
$p_j>0$, for $j = 1, \ldots, m$.
For real numbers $\alpha_\nu$, 
let $\sum_{\nu\in \indexingset} \alpha_\nu X_{n,T_\nu}$ be any arbitrary nontrivial linear combination of these counts (not all $\alpha$'s are 0).
We then have 
$$\frac {\displaystyle \sum_{\nu\in \indexingset}  \alpha_\nu X_{n,T_\nu} - \mu_\collection(n) \, n} {\sigma_\collection(n) \sqrt n} \ \convLaw \ 
\normal(0 ,1),$$
where $\mu_\collection(n)$ and $\sigma_{\collection}^{2}(n)$ are the coefficients of $n$ in the asymptotic expansions for $\E[\sum_{\nu\in \indexingset} \alpha_\nu X_{n,T_\nu}]$ and $\V[\sum_{\nu\in \indexingset} \alpha_\nu X_{n,T_\nu}]$ implicitly given, respectively, by Proposition \ref{Prop:mean} and Theorem \ref{Theo:covar}.\footnote{In our case, the variance
$\sigma^{2}_\collection(n)$ will always be strictly positive. For a more in-depth
  consideration of the variance for shape parameters in random tries,
  see Schachinger~\cite{Schachinger1995}.}
\end{theorem}

\bigskip
In numerous cases, the normality of the univariate linear combination
gives us a multivariate central limit theorem.
Let $\vecX_{n, \collection}$ be the vector
with components $X_{n, T_\nu}$, for $\nu \in \indexingset$. 
A corollary of Theorem~\ref{Theo:main} is that in the aperiodic case we have 
$$\frac {{\vecX_{n, \collection}} - \bfmu_\collection(n) \, n} {\sqrt n} \ \convLaw \ 
\normal_{|\indexingset|}(\bfzero ,\bfSigma_\collection),$$
where $\bfmu_\collection$ is the vector with nonoscillating 
components, that are
the linearity coefficients of the individual means, and
$\normal_{|\indexingset|}(\bfzero ,\bfSigma_\collection)$ 
is the multivariate jointly normally distributed
random vector with mean vector~$\bfzero$ (of $|\indexingset|$ components) 
and the entries 
of $\bfSigma_\collection$ are the nonoscillating linearity coefficients in the variances 
and covariances.\footnote{\label{ourfootnote7}We take an infinite-dimensional 
random vector to have a multivariate normal distribution, when 
every nonzero finite linear combination
of its components has a univariate normal distribution.} 


Let $\{\alpha_\nu \given \nu \in \indexingset\}$ be an arbitrary
collection of real numbers (not all zero). Let
$$Y_{n,\collection} = \sum_{\nu \in \indexingset} \alpha_\nu X_{n, T_\nu};$$
it is our aim to show that, when appropriately 
centered and normalized, $Y_{n,\collection}$ 
converges in distribution to a standard normal random variate
(in the aperiodic case). According to the definition of a multivariate
distribution of an infinite dimensional vector, as given in footnote~\ref{ourfootnote7}, it suffices to consider only 
(arbitrary) finite linear combinations. So, with no loss of generality,
we consider $\indexingset$ finite. The reader will be alerted at
a few places in the sequel, when we switch back to considering an
infinitely countable indexing set. 
\section{A probability space underlying tries}
\label{Sec:ProbabilitySpace}
Our motivation is that any $n$ distinct, infinite-length strings $S_1,
\ldots, S_n$ uniquely define a trie $T_n$.  Each of the $n$ external nodes
corresponds to one of the strings (say $S_i$) as follows:  
The path from the root to the external node corresponds exactly to the
shortest prefix $W_i$ of $S_i$ that is not a prefix of any other $S_j$.
Since we deal with strings of infinite length, however, the potential
overlaps among strings can be arbitrarily long.  Therefore, to
rigorously establish our probability model,
we use a measure-theoretic setup. This has traditionally been
accomplished with an 
approach relying on cylinders; see~\cite{Pittel}. Our methodology of
setting up this probability space is different (our hope is to make the
process more transparent to the reader).

Let $\alphabet_\infty = \prod_{n=1}^{\infty}\alphabet = \alphabet \times \alphabet \times \alphabet \times
\ldots$ denote the set of all infinite-length strings.
We define $\Omega = \prod_{n=1}^\infty A_\infty$.
Each $\omega \in \Omega$ is an infinite-length (ordered) tuple
of infinite-length strings, i.e., 
$\omega = (S_{1},S_{2},S_{3},\ldots)$, where $S_{i} \in A_{\infty}$, 
i.e., each coordinate of $\omega$ is an
infinite-length string of characters from $\alphabet$.


The trie $T(W_{1},\ldots,W_{n})$ induced by a collection
$(W_{1},\ldots,W_{n})$ is defined as 
$$T(W_{1},\ldots,W_{n}) = \{\omega = (S_{1},S_{2},\ldots) \in
\Omega\ |\ \hbox{$W_{j}$ is a prefix of $S_{j}$ for $1\leq j\leq n$}\}.$$
We define the collection of all tries of size $n$ as:
$$\mathcal{T}_{n} = \{T(W_{1},\ldots,W_{n})\ |\ 
\hbox{$(W_{1},\ldots,W_{n})$ has the trie property}\},$$
and then the collection of all tries is $$\mathcal{T} :=
\bigcup_{n=1}^\infty\mathcal{T}_{n}.$$
We say that two tries are disjoint if neither is a subtree of the other.
We use the notation $T \subseteq T^{\prime}$ to indicate that $T$ is a subtree of $T^{\prime}$.

Several remarks help us prepare the setup of the measure space and the
probability measure on this space.
\begin{remark}
There are a countable number of tries.
\end{remark}
\begin{proof}
Since the collection $\alphabet^{*}$ of all finite-length strings
is countable, it follows that there are a countable number of
tuples $(W_{1},\ldots,W_{n})$ satisfying the \emph{trie property}, so
$\mathcal{T}_{n}$ is countable.  Thus, $\mathcal{T}$ is countable too.
\end{proof}
\begin{remark}\label{disjointtrieremark}
For each fixed $n$, the tries in $\mathcal{T}_{n}$ are disjoint.
\end{remark}
\begin{remark}\label{subsettrieremark}
For fixed $m$ and $n$ with $m<n$, if $T \in \mathcal{T}_{m}$ and
$T^{\prime} \in \mathcal{T}_{n}$, then either $T$ and $T^{\prime}$ are
disjoint, or $T^{\prime}\subseteq T$.  Moreover, if
$T^{\prime}\subseteq T$, then $\mbox{\rm height}(T) \le \mbox{\rm height}(T^{\prime})$, where ${\rm height}(T(W_{1},\ldots,W_{n}))$ 
is the length of the longest word among the $W$'s.
\end{remark}
\begin{remark}\label{smallerdisjointcollection}
If $\mathcal{K}$ is a collection of tries, we can use 
Remarks~\ref{disjointtrieremark} and~\ref{subsettrieremark} to replace
$\mathcal{K}$ with another collection of tries $\mathcal{L} \subseteq
\mathcal{K}$ such that $\bigcup_{T\in\mathcal{L}}T =
\bigcup_{T\in\mathcal{K}}T$, and such that the tries in $\mathcal{L}$ 
are disjoint. In fact, $\mathcal{L}$ can be built constructively
from $\mathcal{K}$: 
Organize the tries from $\mathcal{K}$ according to increasing heights.
Only put a trie from $\mathcal{K}$ into 
$\mathcal{L}$, if it is disjoint from all tries of the same-or-lesser
height, as compared to the other tries in $\mathcal{L}$.
(We organize tries by height instead
of numbers of leaves, since there are only a finite number of tries
of each height, but there are an infinite number of tries for each
fixed number $n$ of leaves, with $n\geq 2$.)
\end{remark}
Now denote the set of all countable unions of tries as
$$\mathcal{F} = \Big\{
\bigcup_{T\in\mathcal{K}}T\ |\ \mathcal{K}\subseteq\mathcal{T} \Big\}.$$
\begin{remark}
The collection $\mathcal{F}$ is a $\sigma$-field.
\end{remark}
\begin{proof}  We show (1) $\Omega\in\mathcal{F}$, 
(2) $\mathcal{F}$ is closed under countable unions, and
(3) $\mathcal{F}$ is closed under complementation.

Using $n=1$ and $W_{1}=\varepsilon$ (the trivial string of length 0),
we see $T(W_{1})=\Omega$, so $\Omega \in \mathcal{F}$.
Since each element of $\mathcal{F}$ is a countable union of tries,
$\mathcal{F}$ is closed under countable unions.  
Finally, we show that
$\mathcal{F}$ is closed under complements too.  Consider an element of
$\mathcal{F}$, which necessarily has the form
$\bigcup_{T\in\mathcal{K}}T$ for some collection of tries
$\mathcal{K}\subseteq\mathcal{T}$.  Now define a new collection of
tries, denoted by $\mathcal{K}^{\prime}\subseteq\mathcal{T}$ as follows:  For
each $T^{\prime}\in \mathcal{T}$, let
$T^{\prime}\in\mathcal{K}^{\prime}$ if and only if $T^{\prime}$ is
disjoint from all $T \in \mathcal{K}$.  Then $\bigcup_{T\in\mathcal{K}}T$ and
$\bigcup_{T^{\prime}\in\mathcal{K}^{\prime}}T^{\prime}$ form a
partition of $\Omega$, i.e., they are disjoint, and their union is
exactly $\Omega$.  Therefore,
$\bigcup_{T^{\prime}\in\mathcal{K}^{\prime}}T^{\prime}$ is a countable
union of tries that is exactly the complement of
$\bigcup_{T\in\mathcal{K}}T$.  So, $\mathcal{F}$ is closed under complementation.
(Note: We do not claim $\mathcal{K}
\cup \mathcal{K}^{\prime} = \mathcal{T}$.  
There are generally tries which are neither in $\mathcal{K}$ nor in $\mathcal{K}^{\prime}$.)
\end{proof}
Finally, we define the probability measure on $\mathcal{F}$.  For each $T \in
\mathcal{T}$, we write $T = T(W_{1},\ldots,W_{n})$ for some~$n$ and
some finite-length strings $W_{j}$.  If $W_{j} = a_{i_{1}}a_{i_{2}}\ldots
a_{i_{|W_{j}|}}$, we define $\Prob(W_{j}) = \prod_{k=1}^{|W_{j}|}p_{i_{k}}$.  Then we define
$\Prob(T) = \prod_{j=1}^{n}\Prob(W_{j})$.  Finally, if
$\bigcup_{T\in\mathcal{K}}T \in \mathcal{F}$ for some collection of
tries $\mathcal{K}\subseteq\mathcal{T}$, by
Remark~\ref{smallerdisjointcollection}, we can replace $\mathcal{K}$
with a collection of tries $\mathcal{L}$ such that 
$\bigcup_{T\in\mathcal{L}}T = \bigcup_{T\in\mathcal{K}}T$ and such
that the tries in $\mathcal{L}$ are disjoint.  Thus we define
$$\Prob\Big(\bigcup_{T\in\mathcal{K}}T\Big)
= \Prob\Big(\bigcup_{T\in\mathcal{L}}T\Big)
:= \sum_{T\in\mathcal{L}}\Prob(T).$$

In the sequel,
for all fixed-population models (fixed $n$) the triple $(\Omega, \mathcal{F}, \Prob)$,
with the components just described, will be the probability space on which 
all random variables are defined. For poissonized random variables,
an additional space derived from  $(\Omega, \mathcal{F}, \Prob)$ will
shortly be discussed.
\section{Poissonization}
\label{Sec:poissonization}
Let $\phi_{n, \collection}(u) = \E[e^{u Y_{n, \collection}}]$ be 
the moment generating function 
of the linear combination $Y_{n, \collection}$. We wish 
to asymptotically identify $\phi_{n, \collection}(u)$.
This type of problem is less difficult in the Poisson 
world.
Define the super moment generating function
$$\Phic(u,z) = \sum_{n=0}^{\infty}\phi_{n, \collection}(u)
\frac{z^{n}}{n!}.$$
We interpret the function
$\Phitc(u,z) := e^{-z}\Phic(u,z)$
as a Poisson transform or ``poissonization.''
Indeed, we have
\begin{align*}
\Phitc (u,z) &=  \sum_{n\geq 0}
      \E[e^{u Y_{n, \collection}}]\, {z^{n}\over
                              n!} e^{-z} \\
                       &= \sum_{n\geq 0}\E[e^{u Y_{n, \collection}}]\, \Prob (N_z = n) \\
                       &= \sum_{n\geq 0}\E[e^{u Y_{N_z, \collection }} \, | \, N_z =
                                   n]\, \Prob (N_z = n) \\
                       &= \E\bigl[e^{u Y_{N_z, \collection}}\bigr],
\end{align*}
where
$N_z$ is a random variable with a Poisson 
distribution with mean $z$. 
Thus, $\Phitc(u,z)$ is the  
moment generating function of
a version of $Y_{n, \collection}$ with~$N_z$
replacing the fixed value $n$, transforming 
the view from a fixed population
to a Poisson-distributed population. 
With the Poisson random variable with a large
parameter being highly concentrated about its mean, and the Poisson
model enjoying several convenient independencies in the subtrees,
this poissonization provides an
asymptotic approximation for the moment generating function of $Y_{n, \collection}$
when we take $z =n$.

Likewise, we can study the
poissonized mean and variance then depoissonize them.
For details on depoissonization see~\cite{Jacquet}, and
for a broad discussion see~\cite{Szpankowski}. 

Note that poissonized random variables
should be defined on the product space 
$$(\Omega, \field, \Prob) \times (\mathbb{R}^+, \mathbb{B}^+, \Prob_0) 
=  (\Omega \times \mathbb{R}^+, \field \times  \mathbb{B}^+, \Prob\times \Prob_0) ,$$
where $\mathbb{R}^+$ is the positive real line, $\mathbb{B}^+$ is the usual Borel
sigma field generated by the intervals $(a, b), b> a>0$, and
$\Prob_0$ is the Poisson probability measure.
Later in the paper, we will use analytic continuation to define $Y_{\Nz,\collection}$,
for $z\in\mathbb{C}$.

\section{Proofs} 
\label{Sec:Proofs}
For any given motif $T$ (of size $\tau$), 
we can express $X_{n, T}$  in terms of indicators. 
 We do not impose a condition on the sizes
of the tries in the collection until Section~\ref{Subsec:char}.
Let
$I_{n,  T, w}$ be the indicator that assumes the value~1, if
a random trie with $n$ leaves
contains $T$ as 
a subtree rooted at an internal node of 
the trie
joined to the root of the subtrie 
with a path along which the word $w$ is formed. It is clear that
$$X_{n, T} 
= \sum_{w \in \alphabet^*} I_{n, T, w}.$$
The indicators have the probabilities
\begin{equation}
\Prob(I_{n,T, w} = 1) = {n \choose \tau} \Prob^\tau(w) \bigl(1-\Prob(w)\bigr)^{n-\tau} \Q(T),
\label{Eq:aveindicator}
\end{equation}
where $\Q(T)$ is the shape functional of $T$.  
The linear combination $Y_{n, \collection}$ has the representation
\begin{equation}
\label{Eq:linearcombo}
Y_{n,\collection} 
= \sum_{\nu\in\indexingset} \alpha_\nu \sum_{w \in \alphabet^*} 
I_{n,  T_\nu, w}.
\end{equation}
Subsequently, the poissonized linear combination is
\begin{equation}
\label{Eq:poissonizedlinearcombo}
Y_{N_z,\collection}  = \sum_{\nu\in\indexingset} \alpha_\nu \sum_{w \in \alphabet^*} 
I_{N_z,  T_\nu, w}.
\end{equation}
\subsection{The average of the linear combination}
\label{Subsec:mean}
The linear combination~(\ref{Eq:linearcombo}) has the average
$$\E[Y_{n,\collection}]  =
   \sum_{\nu \in \indexingset} \alpha_\nu 
       \sum_{w \in \alphabet^*} \E\bigl[I_{n, T_\nu, w} \bigr]  = \sum_{\nu\in \indexingset} \alpha_\nu \sum_{w \in \alphabet^*} {n \choose \tau} 
                \, \Prob^\tau(w) \bigl(1-\Prob(w)\bigr)^{n-\tau} \, \Q(T_\nu).$$
To find the average number of occurrences of a certain motif $T$ (of size $\tau$)
in a trie, 
we can take a one-point indexing set $\indexingset = \{1\}$. 
That is, $T_1 = T$ is the only trie in the set. 
Then, with $\alpha_1= 1$, we have 
$$\E[X_{n,T}] =  \sum_{w \in \alphabet^*} {n \choose \tau} 
                \, \Prob^\tau(w) \bigl(1-\Prob(w)\bigr)^{n-\tau} \Q(T).$$
For later reference, we recall here
that the poissonized average is
\begin{align}
\E[X_{N_z, T}] 
&= \sum_{n=0}^\infty \E[X_{n, T}] \, \frac {z^n} {n!} \, e^{-z} \nonumber \\
&=    \frac { e^{-z}} {\tau!}\sum_{w \in \alphabet^*} \Prob^\tau(w) z^\tau
\sum_{n= \tau}^\infty \frac{\bigl(1-\Prob(w)\bigr)^{n-\tau}\, z^{n-\tau}} {(n-\tau)!}\, \Q(T) \nonumber \\
&= \frac {\Q(T)} {\tau!} \, B_1(z) ,
\label{Eq:poissonizedmean}
\end{align}
where
\begin{equation}
B_1  (z)  :=   \sum_{w\in \alphabet^*} \Prob^{\tau}(w)  z^\tau e^{-\Prob(w) z}.
\label{Eq:B1}
\end{equation}
The function $B_1(z)$ has been analyzed in~\cite{Ward}. It has the asymptotic representation
$$B_1(z) = \left({(\tau-2)! \over h} 
        + \xi_{\tau}(z)\right) \, z + o(z),$$
where $\xi_{\tau}(.)$ is an 
oscillating function in the periodic
case, or it is 0 in the aperiodic case. 
Therefore,
$$\E[X_{N_z,T}]  = \Bigl(\frac {\Q(T)} {\tau(\tau-1)h} 
             +   \xi_T (z)
       \Bigr) \, z  + o(z).$$
The result for the mean in Proposition~\ref{Prop:mean} 
follows after depoissonization (see~\cite{Jacquet,Szpankowski}).
\begin{remark}
\label{Rem:Poissonizedindicator}
A useful by-product of the argument is that $\E[I_{N_z,  T, w}] = \Q(T) z^\tau \Prob^\tau (w) e^{-\Prob(w)z}\!/ \tau!$.
\end{remark} 
\subsection{The variance of the linear combination}
\label{Subsec:var}
We cannot derive the variance of the linear combination (\ref{Eq:linearcombo}) via the same straightforward depoissonization argument we utilized to asymptotically
equate $\E[Y_{n,\collection}]$ with $\E[Y_{N_{z},\collection}]$. However, using sharp depoissonization, we \emph{can} obtain the estimate

\begin{equation}
\V[Y_{n,\collection}]=\bigg(\V[Y_{N_{z},\collection}] -z \bigg[{d \over dz} \E[Y_{N_{z},\collection}]\bigg]^{2}\bigg)\bigg|_{z=n}+O(n^{1-\epsilon})\label{Eq:symbolicvariance}
\end{equation}
for some $\epsilon>0$. 
(See~\cite{FuchsTCS,Gaither,FuchsDMTCS} for the details of this
technique; note that the techniques of~\cite{FuchsTCS} could be used
to derive results analogous to those in the present
paper.  The methodology of~\cite{FuchsLee} could probably be used to
establish Theorem~2 as well.)  Equation~(\ref{Eq:symbolicvariance})
implies that to obtain an 
asymptotic expression for $\V[Y_{n,\collection}]$, it will suffice to
derive the asymptotics of both $\V[Y_{N_{z},\collection}]$ and 
${d \over dz} (\E[Y_{N_{z},\collection}])$. We present the steps of the
former calculation in all detail, but leave most of the latter to the
reader; 
they are fairly standard and closely parallel the later stages of the former.

We first obtain an expression for $v(z) :=\V[Y_{N_z, \collection}]$.
It follows from~(\ref{Eq:poissonizedlinearcombo}) that
\begin{align}\V[Y_{N_z, \collection}]&=\sum_{\kappa,\nu \in \indexingset}\alpha_\kappa\alpha_\nu\sum_{w,v \in \Astar}\Cov[I_{N_z,
T_\kappa,w},I_{N_z,T_\nu,v}].\label{indicatorvariance}
\end{align} 
This sum looks daunting to consider in all generality. 
However, as we shall see,
the overwhelming majority of the terms in it will collapse to zero.
We consider four  possible cases for the covariances:
\begin{itemize}
\item [(i)] \textbf{Neither $v$ nor $w$ is a prefix of the other.} 
In this case, 
$\Cov[I_{N_{z},T_\kappa,w},$ $I_{N_{z},T_\nu,v}]=0.$ 
This is a helpful consequence of our working in a Poissonized model---the makeup of the trie at $w$ is 
\emph{independent} of its makeup at $v$ so long as neither $w$ nor $v$ is a prefix of the other. This conclusion holds regardless of whether $\kappa$ and $\nu$ are identical or distinct.\footnote{The same is not true
in the fixed population model. That is, in case (i), $I_{n, T_\kappa,v}$ and $I_{n, T_\nu,w}$ can be dependent.
So, we see the advantage of quickly switching to a Poisson model, rather than transforming
recurrences in the fixed population model.}
\item  [(ii)] \textbf{$v=w$, and $\kappa = \nu$}. 
In this case
we have 
$$\Cov[I_{N_{z},T_\kappa,w},I_{N_z,T_\nu,v}] = \V[I_{N_{z},T_\nu,w}]=  \E[I_{N_z,T_\nu,w}] 
- (\E[I_{N_{z},T_\nu,w}])^{2}.$$
\item  [(iii)] \textbf{$v=w$, and $\kappa \not=\nu$}. 
In this case we have $\E[I_{N_{z},T_\kappa,w}I_{N_z,T_\nu,w}] =0.$
This follows immediately from the nonoverlapping property, which implies that 
$T_\kappa$ and $T_\nu$ cannot both be rooted at the same node. So, in this case
$\Cov[I_{N_z,T_\kappa,w},I_{N_{z},T_\nu,w}]=-\E[\Iwj]\, \E[\Iwi]$.
\item  [(iv)] \textbf{$w$ is a proper prefix of $v$ (or vice-versa)}.
Here, we have $v=wax$, for some $a \in \A$, $x \in \Astar$. Since
$T_{\kappa}$ and $T_{\nu}$ are nonoverlapping, $I_{N_z,T_\kappa,w}$ and $I_{N_z,T_\nu,wax}$ can never simultaneously be $1$, and so we have $\Cov[I_{N_z,T_\kappa,w},I_{N_z,T_\nu,wax}]=-\E[I_{N_z,T_\kappa,w}]\,
\E[I_{N_z, T_\nu ,wax}].$ We note that this result holds even if $\kappa=\nu$.
\end{itemize} \hfil\break
Breaking the covariance expression in (\ref{indicatorvariance}) into these four cases, we obtain
\begin{align}
\V[Y_{\Nz,\collection}]&=\sum_{\nu \in \indexingset}\alpha_\nu^2\sumw \E[I_{N_z,T_\nu,w}]
           -\bigl(\E[I_{N_z, T_\nu, w}]\bigr)^2
\cr
& \qquad {}-\sum_{\substack{\kappa, \nu\in \indexingset \\ \kappa \not =  \nu}}\alpha_\kappa\alpha_\nu \sumw\E[I_{N_z, T_\kappa, w}]\, \E[I_{N_z,T_\nu,w}]\cr
&\qquad {}-2\sum_{\kappa,\nu \in \indexingset}\alpha_\kappa\alpha_\nu\sum_{\substack{w,x \in \Astar \\ a \in \A}}\E[I_{N_z,T_\kappa,w}]\, \E[I_{N_{z},T_\nu,wax}]\cr
&=\sum_{\nu \in  \indexingset}\alpha_\nu^2 
        \sumw \E[I_{N_z,T_\nu,w}] \cr
&\qquad {}-\sum_{\kappa,\nu \in  \indexingset}\alpha_\kappa\alpha_\nu\sumw \E[I_{N_z,T_\kappa,w}]\, \E[I_{N_{z},T_\nu,w}]\cr
&\qquad {} -2\sum_{\kappa,\nu \in  \indexingset}\alpha_\kappa\alpha_\nu\sum_{\substack{w,x \in \Astar \\ a \in \A}}\E[I_{N_z,T_\kappa,w}]\, \E[I_{N_z,T_\nu,wax}].\label{indicatorexpressionforvz}
\end{align}

\subsection{Mellin transform}\label{Mellintransformsection}
Our tool to complete this derivation 
is an integral transform. The Mellin transform of a function $f(x)$ is
$$\int_0^\infty f(x) x^{s-1} \, ds,$$
and will be denoted by $f^*(s)$.
For $s \in \mathbb{C}$, the Mellin transform usually exists in
vertical strips in the complex plane of the form
$$ a < \Re\, s < b,$$
for real numbers $a < b$. We shall denote this strip by $\langle a, b \rangle$.
The function~$f(x)$ can be recovered from its transform by a line integral
$$f(x) = \frac  1 {2 \pi i} \int_{c - i \infty} ^ {c + i \infty} f^*(s) x^{-s} \, ds,$$
for any $c\in (a,b)$. 

At the time of calculating the integral in the inversion, one seeks asymptotic
approximations.
One employs the method of ``closing the box.'' (This method is discussed 
in~\cite{Sortingbook} and~\cite{Szpankowski}.)
In this method, one takes the complex integration over the line $c-iM$ 
and $c+iM$, and then closes 
the box connecting the four corners $c\pm iM$, and $d\pm iM$, 
for an arbitrary $d>c$.
The number $M$ is chosen in such a way that no pole is crossed.
Cauchy's Residue Theorem gives
$$\lim_{M \to \infty} \oint f^*(z)z^{-s} \, ds = 2 \pi i  \sum
\mbox{residues of poles in}\, \langle c, d\rangle.$$
The contour integral can be written as
$$
\oint f^*(s)z^{-s} ds = \int_{c-iM}^{d-iM} + \int_{d-iM}^{d+iM} + \int_{d+iM}^{c+iM}+ \int_{c+iM}^{c-iM}.
$$
In the context of random structures, the Mellin transform often includes gamma functions.
In this context, when we let $M\to \infty$, the line integrals at the top and bottom sides of the box approach 0,
as the magnitude of the gamma function decreases exponentially fast with its imaginary part.
Moreover, the integral at the right side of the box 
introduces an error term of the order $O(z^{-d})$.
Hence, we have 
$$ f(z)  =  O(z^{-d}) -  \sum \mbox{residues of poles in}\, \langle c, d\rangle.$$
The problem has now been reduced to residue computation.
See~\cite{Flajolet} for a survey on the use of
the Mellin transform in the analysis of random structures and algorithms,
where the reader can find detailed discussions on the procedure
and other standard tricks of the trade.

By (\ref{indicatorexpressionforvz}) and Remark~\ref{Rem:Poissonizedindicator},
we can now write
$v(z)=v_{1}(z)-v_{2}(z)-2v_{3}(z)$, where 
\begin{align*}
v_1(z)&=\sum_{\nu \in \indexingset}\alpha_\nu^2 {\Q(T_\nu) \over \tau_\nu!} \sumw z^{\tau_\nu}\Prob^{\tau_\nu}(w)e^{-z\Prob(w)},\\
v_2(z)&=\sum_{\kappa,\nu \in \indexingset} \alpha_\kappa\alpha_\nu {\Q(T_\kappa)\, \Q(T_\nu) \over \tau_\kappa!\, 
\tau_\nu!} \sumw z^{\tau_\kappa+\tau_\nu}\Prob^{\tau_\kappa+\tau_\nu}(w) e^{-2z\Prob(w)},\\
v_{3}(z)&=\sum_{\kappa,\nu \in  \indexingset}\alpha_\kappa\alpha_\nu{\Q(T_\kappa)\, \Q(T_\nu) \over \tau_\kappa!\, \tau_\nu!}\sum_{\substack{w,x \in \Astar \\ a \in \A}}z^{\tau_\kappa+\tau_\nu}\, \Prob^{\tau_\kappa+\tau_\nu}(w)\Prob^{\tau_\nu}(ax)e^{-z\Prob(w)(1+\Prob(ax))}.
\end{align*}
Here $\tau_{\nu}$ simply denotes the size of $T_{\nu}$.
\subsection{Fundamental strip}\label{Fundamentalstripsection}
Now we want to take Mellin transforms. It is easy to establish a left-hand boundary for our fundamental 
strip.  Let $\tau=\min\{\tau_\nu\given \nu \in \indexingset\}$, then, as $z\rightarrow 0$, 
we have $v_1(z)=O(z^{\tau})$, and $v_{2}(z)$ and $v_{3}(z)$ are both $O(z^{2{\tau}})$. 
Barring the trivial motifs which are either empty
or have only one string (and therefore create no kind of splitting),
we now consider the motifs with at least two strings.
Since a nontrivial motif has at least two strings for 
its construction, then ${\tau} \ge 2$,
and we have $v(z)=O(z^{2})$, as $z \to 0$. We also note that no member
of a nonoverlapping 
collection of motifs can be of size one or less, as such motifs are overlapping with every
other possible motif.

Finding the right-hand Mellin boundary requires a bit more work. 
We note that $v_{1},v_{2}$ and $v_{3}$ are all $O\Bigl(\sumw (z\Prob(w))^je^{-z\Prob(w)}\Bigr)$, 
as $z \rightarrow \infty$, for some $j \ge 2$ (either $j= \tau_\nu$, 
or $j=\tau_\kappa+\tau_\nu$.) 
For any $\varepsilon \in (0,1)$, we can write 
\begin{equation}
\label{rightmellinbound}
\sumw (z\Prob(w))^je^{-z\Prob(w)}=\sumw \Prob^{1+\varepsilon}(w)\times 
           z^j\Prob^{j-(1+\varepsilon)}(w)e^{-z\Prob(w)}.
\end{equation}
Doing minimization by 
calculus on the expression $z^j\Prob^{j-(1+\varepsilon)}(w)e^{-z\Prob(w)}$, 
with $\Prob(w)$ as the variable, we find that 
$$z^j\Prob^{j-(1+\varepsilon)} (w) e^{-z\Prob(w)} \le z^{1+\varepsilon}(j-1-\varepsilon)^{j-1-\varepsilon}e^{-j+1+\varepsilon}=O(z^{1+\varepsilon}).$$
So, (\ref{rightmellinbound}) implies that
$$\sumw (z\Prob(w))^je^{-z\Prob(w)} = \sumw\Prob^{1+\varepsilon}(w)O(z^{1+\varepsilon}) = O(z^{1+\varepsilon}).$$
Therefore, for each $j = 1, 2, 3$,
we have $v_j(z)=O(z^{1+\varepsilon})$, for every $\varepsilon>0$, as $z \to \infty$.
We have now shown that $\langle -2,-1\rangle$ is a valid fundamental strip for~$v(z)$. 
\subsection{Asymptotics of $v(z)$}\label{AsymptoticsofVsection}
We now take the Mellin transform of $v(z)$. Recalling that $v(z)=v_1 (z)-v_2(z)-2v_3(z)$, we extract the asymptotics of $v(z)$ one piece at a time. We have
$$\vos(s)=\sum_{\nu \in  \indexingset}\alpha_\nu^{2} {\Q(T_\nu) \over \tau_\nu!} \sumw \Prob^{-s}(w\, )\Gamma
               (s+\tau_\nu)
     =  \sum_{\nu \in  \indexingset}\alpha_\nu^{2} {\Q(T_\nu) \over \tau_\nu!} \times
\frac {\Gamma(s+\tau_\nu)} {1-\sum_{j=1}^m p_j^{-s}}.$$ 
Invoking the closing-the-box method, after a residue calculation we 
find the inverse Mellin transform: 
\begin{align*}
v_{1}(z)={z \over h}\sum_{T_\nu \in \indexingset}\alpha_\nu^2\,
{\Q(T_\nu) \over \tau_{\nu}(\tau_{\nu}-1)} + \delta_{\collection,1}(z) z+o(z),
\end{align*}
where $\delta_{\collection, 1}$ is possibly fluctuating 
 with average value zero,
when the probability set is aperiodic, and is $0$ otherwise. 
In the periodic case, 
we can improve the $o(z)$ bound to $O(z^{1-\varepsilon})$, 
for some $0<\varepsilon<1$.

 A comprehensive discussion about the difference between
  the periodic 
and aperiodic cases is given at length in~\cite{Flajolet2010}.  For
readers who want to understand the nuances of these case, this
treatise is strongly recommended.

Regarding $v_2(z)$, we have
\begin{align*}
\vts(s)&=\sum_{\kappa, \nu \in  \indexingset}\alpha_\kappa\alpha_\nu\,  {\Q(T_\kappa)\, \Q(T_\nu) \over \tau_\kappa!\, \tau_\nu!} \sumw 2^{-s-\tau_\kappa-\tau_\nu}\Prob(w)^{-s}\Gamma(s+\tau_\kappa+\tau_\nu)\\
&=\sum_{\kappa, \nu \in \indexingset}\alpha_\kappa\alpha_\nu {\Q(T_\kappa)\, \Q(T_\nu) \over \tau_\kappa!\, 
\tau_\nu!}2^{-s-\tau_\kappa-\tau_\nu} \, \frac{\Gamma(s+\tau_\kappa+\tau_\nu)} {1-\sum_{j=1}^m p_j^{-s}}.
\end{align*}

After a residue calculation, we 
find the inverse Mellin transform:
$$v_2(z) = {z\over h}\sum_{\kappa, \nu\in 
         \indexingset}\alpha_\kappa\alpha_\nu\, {\Q(T_\kappa)\, 
\Q(T_\nu) \over (\tau_\kappa+\tau_\nu)(\tau_\kappa+\tau_\nu-1)}\, 2^{1-\tau_\kappa-\tau_\nu}{\tau_\kappa+\tau_\nu \choose \tau_\nu}
         +\delta_{\collection, 2}(z) z+o(z),$$
where $\delta_{\collection, 2}$ is possibly fluctuating 
with average value zero,
when the probability set is aperiodic, and is $0$ otherwise. 
In the periodic case, 
we can improve the $o(z)$ bound to $O(z^{1-\varepsilon})$, 
for some $\varepsilon>0$.

The Mellin of transform of $v_3(z)$ is the most complicated to calculate. We first note that we have  
\begin{align}
v_3(z)
&=\sum_{\kappa,\nu \in \indexingset}\alpha_\kappa\alpha_\nu{\Q(T_\kappa)\, \Q(T_\nu) \over \tau_\kappa!\, \tau_\nu!}\sum_{\substack{w,x \in \Astar \\ a \in \A}} z^{\tau_\kappa+\tau_\nu}\, \Prob^{\tau_\kappa+\tau_\nu}(w)e^{-z\Prob(w)} \, \Prob^{\tau_\nu}(ax)e^{-z\Prob(w)\Prob(ax)}\cr
&=\sum_{\kappa,\nu \in \indexingset}\alpha_\kappa\alpha_\nu{\Q(T_\kappa)\, \Q(T_\nu) \over \tau_\kappa!\, \tau_\nu!}\sum_{\substack{w,x \in \Astar \\ a \in \A}}z^{\tau_\kappa+\tau_\nu}\Prob^{\tau_\kappa+\tau_\nu}(w)e^{-z\Prob(w)}\Prob^{\tau_\nu} (ax)  \sum_{j=0}^\infty{(-z\Prob(w)\, \Prob(ax))^j \over j!}.\label{domconv}
\end{align}
We can carry the Mellin transform inside the innermost sum (recall that
we are dealing with finite indexing sets), yielding
$$\vhs(s) =\sum_{\kappa,\nu \in \indexingset} \alpha_\kappa\alpha_\nu\, {\Q(T_\kappa)\, \Q(T_\nu) \over \tau_\kappa!\, \tau_\nu!}   \sum_{\substack{w,x \in \Astar \\ a \in \A}}\sum_{j= 0}^\infty {(-1)^j \Prob^{\tau_\nu+j}(ax)\, \Prob^{-s}(w)\, \Gamma(s+\tau_\kappa+\tau_\nu+j) \over j!}.$$
Now we approximate the inverse Mellin by closing the box and considering residues. The Gamma functions are all analytic in $\langle -2,-1\rangle$, since $\tau_\kappa,\tau_\nu \ge 2$. So, 
all the singularities come from $\sum_{w \in \Astar}\Prob^{-s}(w)$. Taking the inverse Mellin transform, we have 
\begin{align*}
v_3(z) &={z \over h}\sum_{\kappa, \nu \in \indexingset}\alpha_\kappa\alpha_\nu{\Q(T_\kappa)\, \Q(T_\nu) \over \tau_\kappa\, !\, \tau_\nu!}\sum_{\substack{x \in \Astar \\ a \in \A}}\sum_{j=0}^\infty {(-1)^j \Prob^{\tau_\nu+j}(ax)\, (\tau_\kappa+\tau_\nu + j-2)! \over j! }+\delta_{\collection,3}
(z) z+o(z)\\
&= {z \over h}\sum_{\kappa, \nu \in \indexingset}\alpha_\kappa\alpha_\nu{\Q(T_\kappa)\, \Q(T_\nu) \over \tau_\kappa!\, \tau_\nu!}\sum_{j= 0}^\infty (-1)^j {\sum_{k=1}^m p_k^{\tau_\nu+j} \over 1 - \sum_{k=1}^mp_{k}^{\tau_\nu+j}}
     {(\tau_\kappa + \tau_\nu+j-2)! \over j!}+\delta_{\collection,3}
(z) z+o(z),
\end{align*}
where $\delta_{\collection, 3}$ is  possibly fluctuating 
with average value zero,
when the probability set is aperiodic, and is $0$ otherwise. 
In the periodic case, 
we can improve the $o(z)$ bound to $O(z^{1-\varepsilon})$, 
for some $0< \varepsilon< 1$.

To summarize, we have
\begin{align*}
v_1(z) &={z \over h}\sum_{\nu \in \indexingset}\alpha_\nu^{2} {\Q(T_\nu) \over \tau_\nu(\tau_\nu-1)}\, +\delta_{\collection, 1}(z)z+o(z),\\
v_2(z)&={z \over h}\sum_{\kappa, \nu \in \indexingset}\alpha_\kappa\alpha_\nu \, {\Q(T_\kappa)\, \Q(T_\nu) \over  (\tau_\kappa+\tau_\nu)(\tau_\kappa+\tau_\nu-1)}2^{1-\tau_\kappa-\tau_\nu}{\tau_\kappa+\tau_\nu\choose \tau_\nu}+ \delta_{\collection, 2} (z)z+o(z),\\
v_3(z)&= {z \over h}\sum_{\kappa, \nu \in \indexingset}\alpha_\kappa\alpha_\nu\, {\Q(T_\kappa)\, \Q(T_\nu) \over \tau_\kappa!\, \tau_\nu!}\sum_{j= 0}^\infty (-1)^j {\sum_{k=1}^m p_k^{\tau_\nu+j} \over 1 - \sum_{k=1}^m p_k^{\tau_\nu+j}} \times
 {(\tau_\kappa+\tau_\nu+j-2)! \over j!}  + \delta_{\collection, 3}
(z) z+o(z).
\end{align*}
Since $v(z)= v_{\collection, 1}(z)-v_{\collection, 2} (n)-2v_{\collection, 3} (z)$, we have
\begin{align*}
v(z)&={z \over h}\sum_{\nu \in \indexingset}\alpha_\nu^{2} {\Q(T_\kappa) \over \tau_\nu(\tau_\nu-1)}\, -{2z \over h}\sum_{\kappa, \nu \in \indexingset}\alpha_{\kappa}\alpha_{\nu}{\Q(T_\kappa)\Q(T_\nu) \over \tau_\kappa!\, \tau_\nu!}\\
& \qquad \times \bigg[2^{-\tau_\kappa -\tau_\nu}(\tau_\kappa+\tau_\nu-2)!
+\sum_{j = 0}^\infty (-1)^j {\sum_{k=1}^m p_k^{\tau_\nu+j} \over 1 - \sum_{k=1}^m p_k^{\tau_\nu+j}} \times {(\tau_\kappa + \tau_\nu+j-2)! \over j!}\bigg]+ \delta_\collection (z)z+o(z),
\end{align*}
where $\delta_\collection(t):=\delta_{\collection, 1}(t)-\delta_{\collection, 2}(t)-2\delta_{\collection, 3}(t).$

We recall that our expression~(\ref{Eq:symbolicvariance}) 
for $\V[Y_{n,\collection}]$ includes $-z \left({d \over dz} \E[Y_{N_{z},\collection}]\right)^{2}$. We must, therefore, calculate the asymptotics of ${d \over dz} \E[Y_{N_{z},\collection}]$. Taking the derivative 
of~(\ref{Eq:poissonizedmean}) and summing over all motifs in $\collection$, we obtain
$${d \over dz} \E[Y_{N_{z},\collection}]=\sum_{\nu \in \indexingset}\alpha_{\nu} {\Q(T_{\nu}) \over \tau_{\nu}!} \sumw \tau_{\nu} z^{\tau_{\nu}-1}\Pw^{\tau_{\nu}}e^{-z\Pw}- z^{\tau_{\nu}}\Pw^{\tau_{\nu}+1}e^{-z\Pw}.$$
From here, using the same techniques we employed to discover the asymptotics of $v(z)$, we find that
$${d \over dz}\E[Y_{N_{z},\collection}]={1 \over h}\sum_{T_{\nu} \in \collection}\alpha_{\nu} {\Q(T_{\nu}) \over \tau_{\nu}(\tau_{\nu}-1)}+\delta_{\collection,4}(z)+o(1).$$

The result of Theorem 2 now follows from depoissonization 
(again see~\cite{Jacquet,Szpankowski}), 
i.e.~by substituting our calculated expressions into (\ref{Eq:symbolicvariance}). Making the substitutions, we find that
\begin{align}\label{Eq:LinearVariance}
\V[Y_{n,\collection}]&={n\over h}\bigg(\sum_{\nu \in \indexingset}\alpha_\nu^{2} {\Q(T_\kappa) \over \tau_\nu(\tau_\nu-1)}\, -2\sum_{\kappa, \nu \in \indexingset}\alpha_{\kappa}\alpha_{\nu}{\Q(T_\kappa)\Q(T_\nu) \over \tau_\kappa!\, \tau_\nu!}\cr
& \qquad\qquad{} \times \bigg[2^{-\tau_\kappa -\tau_\nu}(\tau_\kappa+\tau_\nu-2)!
+\sum_{j = 0}^\infty (-1)^j {\sum_{k=1}^m p_k^{\tau_\nu+j} \over 1 - \sum_{k=1}^m p_k^{\tau_\nu+j}}{(\tau_\kappa + \tau_\nu+j-2)! \over j!}\bigg]\bigg) + \delta_\collection (n)n\cr
& \qquad{}-n\bigg({1 \over h}\sum_{T_{\nu} \in \collection}\alpha_{\nu} {\Q(T_{\nu}) \over \tau_{\nu}(\tau_{\nu}-1)}+\widehat{\delta}_{\collection}(n)\bigg)^{2}+o(n).
\end{align}

We note that in the periodic case, the error term $O(n^{1-\epsilon})$ in our symbolic variance expression (\ref{Eq:symbolicvariance}) will survive the depoissonization process, so that our error in (\ref{Eq:LinearVariance}) will improve to $O(n^{1-\epsilon})$. In the aperiodic case, however, the $O(n^{1-\epsilon})$ bound will be subsumed by the coarser $o(z)$ estimates for our poissonized quantities.
\subsection{Covariance structure}
\label{Subsec:Covar}
To find the variance of the number of occurrences of an individual 
motif $T$ (of size $\tau$)
in a trie, again we take a one-point indexing set $\indexingset =
\{1\}$, with trie $\alpha_1 = 1$, and $T_1 = T$. So,
\begin{align*}
\V[X_{n, T}] &= \biggl[{\Q(T) \over \tau (\tau-1)h}\, 
- {2 \Q^2(T) \over  h}\bigg(\frac{2^{-2\tau}} {2\tau(2\tau-1) } {2\tau
  \choose \tau} +\frac1  {(\tau!)^2}\sum_{j = 0}^\infty (-1)^j
{\sum_{k=1}^m p_k^{j+\tau} \over 1 - \sum_{k=1}^m p_k^{j+\tau}} \times
{(j+ 2\tau-2)! \over j!}\bigg)\\
&\qquad{}+ \delta_{T}(n) - \Big({\Q(T) \over \tau(\tau-1)h} + \widehat \delta_T(n)\Big)^2\biggr]\, n +o(n).
\end{align*}

To compute the covariance between two nonoverlapping
tries $T$ and $\widetilde T$, we take a collection $\collection = \{T, \widetilde T\}$
comprised of only these two trees (of sizes $\tau$ and $\widetilde \tau$, respectively)
and consider the linear combination
$X_{n,T} + X_{n,\widetilde T}$. In this manner we arrive at an analogous expression for $\V[X_{n,T} + X_{n,\widetilde T}].$
Then from the standard relation
$$\V\bigl[X_{n,T} + X_{n,\widetilde T}\bigr] = \V\bigl[X_{n,T} \bigr] 
            + \V\bigl[ X_{n,\widetilde T}\bigr] + 2\, \Cov\bigl[X_{n,T} , X_{n,\widetilde T}\bigr],$$
and the already computed variance,
we find
the covariance

\begin{align*}
\Cov[X_{n,T} , X_{n,\widetilde T}]
&=\bigg[-{2 \Q(T)\, \Q(\widetilde T) \over \tau!\, \widetilde \tau!\,
  h}  \bigg(2^{-\tau -\widetilde\tau}(\tau+\widetilde \tau-2)!\\
&\qquad\qquad\qquad {}+2^{-1}\sum_{j = 0}^\infty (-1)^j\bigg( {\sum_{k=1}^m p_k^{\tau+j} \over 1 - \sum_{k=1}^m p_k^{\tau+j}}+ {\sum_{k=1}^m p_k^{\widetilde\tau+j} \over 1 - \sum_{k=1}^m p_k^{\widetilde\tau+j}}\bigg) \times {(\tau+ \widetilde \tau+j-2)! \over j!}\bigg) \cr
& \qquad\qquad  {} + \frac 1 2\bigl( \delta_{T, \widetilde T}(n) - \delta_T (n) - \delta_{\widetilde T} (n)\bigr)   - {\Q(T)\, \Q(\widetilde T) \over  \tau(\tau-1)\widetilde\tau(\widetilde\tau-1)h^{2}}  \cr
& \qquad \qquad {}+ {1 \over 2h}\bigg({\Q(T) \over \tau(\tau-1)}\bigl(\widehat \delta_{T}(n)-\widehat\delta_{T,\widetilde T}(n)\bigr) + {\Q(\widetilde T) \over \widetilde \tau(\widetilde \tau-1)} \bigl(\widehat\delta_{\widetilde T}(n)-\widehat\delta_{T,\widetilde T}(n)\bigr)\bigg)\cr
&\qquad\qquad {} -(\widehat\delta_{T,\widetilde T}(n)^{2}- \widehat\delta_{T}(n)^{2}-\widehat\delta_{\widetilde T}(n)^{2})
\bigg]\, n+o(n),
\end{align*}
\subsection{The moment generating function of the linear combination}
\label{Subsec:char}
Until now, we have not imposed a condition on the  sizes
of the tries in the collection. However, arguments for the limit distribution
go more smoothly, if we assume the tries in the collection all have the same
size. 

Let $\collection$ be a collection of tries
all having the same size $\tau$. 
The moment generating function $\phi_{n, \collection}(u)$ of~$Y_{n, \collection}$
can be developed recursively.
It is clear that
when $n > \tau$, we do not have a tree with any shape 
from~$\collection$ starting at the root of the trie.  Thus, for $n >
\tau$, we have a recurrence, obtained by conditioning on $N_1, \ldots, N_m$, 
the sizes of the subtrees, and following the first letter in each word, which
is namely
\begin{align*}
\phi_{n, \collection}(u)
&= \E\bigl[e^{u  Y_{n, \collection}} \bigr] \\
&= \sum_{n_{1} + \cdots+n_{m}=n}  \E\Bigl[
  \exp\Bigl( u \sum_{\nu \in \indexingset}  \alpha_\nu X_{n,T_\nu}\Bigr) 
           \given  N_{1}=n_{1},\ldots,N_{m}=n_{m} \Bigr]\ \Prob(N_{1}=n_{1},\ldots,N_{m}=n_{m})   \\
&= \sum_{n_{1}+\cdots+n_{m}=n} 
    \E\Bigl[
  \exp\Bigl( u \sum_{\nu \in \indexingset} \alpha_\nu\sum_{j=1}^m X_{n_j, T_\nu} \Bigr) \Big]  {n\choose n_{1},\ldots,n_{m}}\,
   p_{1}^{n_{1}}\ldots p_{m}^{n_{m}}  \\
&=  \sum_{n_{1}+\cdots+n_{m}=n} \E\big[\exp\big(uY_{n_{1}, \collection} + \cdots 
   + uY_{n_{m}, \collection} \big)\big] {n\choose n_{1},\ldots,n_{m}}\,
   p_{1}^{n_{1}}\ldots p_{m}^{n_{m}}. 
\end{align*}
By the independence in the subtrees, we can write, for $n > \tau$,

\begin{equation}
\phi_{n, \collection}(u) = n!\, 
\sum_{n_{1}+\cdots+n_{m}=n}
\prod_{j=1}^{m}\frac{\phi_{n_{j}, \collection}(u)\, p_{j}^{n_{j}}}{n_{j}!}.
\label{Eq:nmorethantau}
\end{equation}
Note that, for $n < \tau$, we have $Y_{n, \collection} = 0$, so
\begin{equation}
\label{Eq:nlessthantau}
\phi_{n, \collection} (u) = \E\bigl[e^{u Y_{n, \collection }}\bigr] = \E[e^0] = 1, \qquad\hbox{for \ } n <
  \tau.
\end{equation}

\begin{lemma}
\label{Lm:poissonizedmgf}
The poissonized moment generating 
 function $\Phic := e^{z}\E[e^{u Y_{N_z,
       \collection}}]$ satisfies the recurrence
\begin{align}\label{phibound}
\Phitc(u,z) = \bigl(\phi_{\tau, \collection}(u)-1\bigr)\Big(1-\sum_{j=1}^{m}p_j^{\tau}\Big)\, \frac{z^{\tau}}{\tau!} \, e^{-z}
          + \prod_{j=1}^{m} \Phitc(u,p_j z).\end{align}
\end{lemma}

\begin{proof}
We compute $\Phic(u,z)$ from~(\ref{Eq:nmorethantau}) 
and the boundary conditions~(\ref{Eq:nlessthantau}).  
We get
\begin{align}
\Phitc(u,z)
&= e^{-z}\sum_{n=0}^{\tau-1} \frac{z^{n}}{n!}
+ \phi_{\tau, \collection}(u)\frac{z^{\tau}}{\tau!}\, e^{-z} + \sum_{n=\tau+1}^{\infty}\sum_{n_{1}+\cdots+n_{m}=n}
\prod_{j=1}^{m}\frac{\phi_{n_j, \collection}(u) p_j^{n_j}z ^{n_j}e^{-p_j z}}{n_{j}!}\nonumber \\
&= e^{-z}\sum_{n=0}^{\tau-1}\frac{z^{n}}{n!}
+ \phi_{\tau, \collection}(u)\frac{z^{\tau}}{\tau!}\, e^{-z} + e^{-z}\sum_{n=0}^{\infty}\sum_{n_{1}+\cdots+n_{m}=n}
\prod_{j=1}^{m}\frac{\phi_{n_j, \collection}(u)(p_{j}z)^{n_{j}}}{n_{j}!} \label{Eq:penultimate}\\
& \qquad {} - e^{-z}\sum_{n=0}^{\tau}\sum_{n_{1}+\cdots+n_{m}=n}
\prod_{j=1}^{m}\frac{{\phi_{n_j, \collection}(u)}(p_{j}z)^{n_{j}}}{n_{j}!}. \nonumber
\end{align}
For $n < \tau$, the solution of the equation $n_1 + \cdots + n_m =n$
in nonnegative integers yields nonnegative integers~$n_j$ that are all less than
$\tau$, with corresponding $\phi_{n_j, \collection} (u) = 1$, for $j = 1, \ldots, m$.
In this case,
the product in~(\ref{Eq:penultimate}) of the previous display
becomes 
$$\prod_{j=1}^{m}\frac{\phi_{n_j, \collection}(u)(p_{j}z)^{n_{j}}}{n_{j}!} 
= \prod_{j=1}^{m}\frac{(p_{j}z)^{n_{j}}}{n_{j}!}.$$
In the case $n = \tau$, we have two cases: 
\begin{itemize}
\item [(i)] The integer solution gives 
all variables equal to 0, except $n_r$, for some $1 \le r \le m$,
which must be equal to $\tau$. 
In this case we have
\begin{align*}
\prod_{j=1}^{m}\frac{\phi_{n_j, \collection}(u)(p_{j}z)^{n_{j}}}{n_{j}!} 
= \sum_{r=1}^m \Bigl(\prod_{j=1  \atop {j \not = r}}^{m}
   \frac{\phi_{0, \collection}(u)(p_{j}z)^0}{0!} \Bigr)
   \times \frac {\phi_{\tau, \collection} (u) (p_r z)^{\tau}} {\tau!}
=  \frac {\phi_{\tau, \collection} (u) z^{\tau}} {\tau!} \sum_{r=1}^m p_r^{\tau}.
\end{align*}
\item [(ii)] The integer solution gives 
all variables $n_j$ less than $\tau$, yielding
$$\prod_{j=1}^{m}\frac{\phi_{n_j, \collection}(u)(p_{j}z)^{n_{j}}}{n_{j}!} 
= \prod_{j=1}^{m}\frac{(p_{j}z)^{n_{j}}}{n_{j}!}.$$
\end{itemize}
The following calculation ensues:
\begin{align*}
\Phitc(u,z) &= e^{-z}\sum_{n=0}^{\tau-1}\frac{z^{n}}{n!}
   + \phi_{\tau, \collection}(u)\frac{z^{\tau}}{\tau!} e^{-z}
   + \prod_{j=1}^{m}\Phitc(u,p_{j}z)\cr
&\qquad {} - e^{-z}\sum_{n=0}^{\tau}\frac{z^{n}}{n!}\sum_{n_{1}+\cdots+n_{m}=n}
   {n\choose n_{1},\ldots,n_{m}}\prod_{j=1}^{m}p_{j}^{n_{j}} \cr
&\qquad {} + e^{-z}\sum_{j=1}^m\frac{z^{\tau}}{\tau!}
   p_{j}^{\tau}-e^{-z}\frac{\phi_{\tau, \collection}(u)z^{\tau}}{\tau!}\sum_{j=1}^{m}p_{j}^{\tau}.
\end{align*}
By the multinomial theorem,
the sum involving the multinomial coefficients is 1, and we get
\begin{align*}              
\Phitc(u,z)
&= e^{-z} \sum_{n=0}^{\tau-1}\frac{z^{n}}{n!}
  + \phi_{\tau, \collection}(u)\frac{z^{\tau}}{\tau!}e^{-z}
  + \prod_{j=1}^{m}\Phic(u,p_{j}z) 
  - e^{-z}\sum_{n=0}^{\tau}\frac{z^{n}}{n!} + \frac{(1-\phi_{\tau, \collection}(u))z^{\tau}}{\tau!} e^{-z}\sum_{j=1}^{m}p_{j}^{\tau}\cr   
&= (\phi_{\tau, \collection}(u)-1)\frac{z^{\tau}}{\tau!}
  \Big(1-\sum_{j=1}^{m}p_{j}^{\tau}\Big)e^{-z} + \prod_{j=1}^{m}\Phitc(u,p_{j}z).
\end{align*}
\end{proof}
\subsection{Limit distributions}
\label{Susbec:limit}
Our final task is to prove Theorem~\ref{Theo:main}, which states that after centralization and normalization the linear combination $Y_{n,\collection}$ converges in distribution to the standard normal distribution. For this job we use a powerful result from Jacquet and Szpankowski (adapted to our purposes) which is specifically formulated for CLT-type arguments which involve poissonization. 

\begin{lemma}\label{Lm:distdepoi}[Jacquet and Szpankowski, 1998]
Let $W_{n}$ be a random variable and  $\Gt(u,z)=\E[e^{uW_{N_{z}}}]$
its poissonized moment generating function. 
Consider $u$ in a fixed interval on the real line, centered at the
origin (the values of the constants depend on the length of this fixed interval).
Suppose further that 
there exist values $\epsilon>0$, $A>0$, $B>0$, and $R>0$, such that 
the following conditions hold:
\begin{enumerate}
\item We have
$$
|W_{n}| \leq Cn, \qquad 
\E[W_{N_{z}}] = \widetilde{\mu}(z)z + o(z), \qquad\hbox{and}\qquad 
\V[W_{N_{z}}] = \widetilde{\sigma}^{2}(z)z + o(z),
$$
for some fixed constant $C>0$ and some bounded functions $\widetilde{\mu}(z)$ and $\widetilde{\sigma}^{2}(z)$.
\item In the cone $\mathcal{C}(\epsilon) =\{z=x+iy\;:\;|y| \leq
  x^{1-\epsilon}\}$, we have the bound
$$
\big|\log\bigl(\Gt(u,z)\bigr)\big| \leq B
|z|^{1+\epsilon},\qquad\hbox{when $|z|>R$}.
$$
\item Let $V_{n}:=\V[W_{n}]$.  
Outside the cone $\mathcal{C}(\epsilon)$, when $|z| = n$, we have the bound
$$
\Big|e^{z}\Gt\Big(\frac{u}{\sqrt{V_{n}}},\;z\Big)\Big| \leq \exp(n-An^{1/2+\epsilon}),
$$
for sufficiently large $n$. 
\end{enumerate}
Then the random variable $(W_{n}-\widetilde{\mu}(n))/\sqrt{V_{n}}$ converges in distribution to a standard normal.
\end{lemma}
To prove Theorem~\ref{Theo:main}, it suffices to show that the linear combination $Y_{n,\collection}$ satisfies the conditions  of Lemma~\ref{Lm:distdepoi}, with $\Phitc(u,z)$ playing the role of $\Gt(u,z)$. Our proof of this parallels the argument given in \cite{Jacquet}.

We have already proved that Condition 1 holds for $Y_{N_{z},\collection}$, as we computed its mean and variance en route to proving Proposition~\ref{Prop:mean} and Theorem~\ref{Theo:covar}. The requirement that $|Y_{n,\collection}| \leq Cn$ follows from the fact that $Y_{n,\collection} = \sum_{\nu \in \indexingset} \alpha_\nu X_{n, T_\nu}$ where $\indexingset$ is assumed to be finite (see the paragraph at the end of Section \ref{Sec:Results}), and we know that each $ X_{n, T_\nu} \leq n$ since a given motif can occur at most $n$ times in a trie of size $n$.

Regarding Condition 2, we note that the assumption that $|y| \leq x^{1-\epsilon}$ implies that $|z| \leq x\sqrt{1+x^{-2\epsilon}} \leq x(1+x^{-\epsilon})$. From there we can conclude that 
\begin{align*}
\big|e^{-z}\big|=e^{-x} \leq \exp\Big(-\frac{|z|}{1+x^{-\epsilon}}\Big).
\end{align*}
Plugging this bound into the definition of $\Phitc$, we obtain
\begin{align*}
\big|\Phitc(u,z)\big| \leq |e^{-z}| \sum_{n \geq 0}\frac{|ze^{uC}|^{n}}{n!} \leq \exp\Big(-\frac{|z|}{1+x^{-\epsilon}}+|z|e^{uC}\Big),
\end{align*}
from which Condition 2 readily follows. (In the first inequality, we used the hypothesis that $|Y_{n}| \leq Cn$.)

Condition 3 is the most interesting to verify. Our device (inspired by \cite{Jacquet}) will be to induct over a sequence of nested domains 
\begin{align*}
D_{k}=\{z\;:\; \xi \leq |z| \leq  \xi \lambda^{k}\},
\end{align*}
where $\xi>0$ and $1 < \lambda < \big(\max_{1\le j\le  m}\{p_{j}\}\big)^{-1}$ are fixed quantities. We note that whenever $z \in D_{k+1}$, we have $p_{j}z \in D_{k}$ for every $j$.

The recurrence~(\ref{phibound}) from Lemma \ref{Lm:poissonizedmgf} lies at the heart of our methodology. Unfortunately~(\ref{phibound}) is phrased in terms of a product, and our technique requires a sum. We circumvent this problem by taking the log of~(\ref{phibound}). Before doing this, however, we must rewrite its right-hand side as a product (after first multiplying through by $e^{z}$):
\begin{align*}
\Phic(u,z) &=  \Big(\prod_{j=1}^{m}\Phic(u,p_{j}z)\Big) \times \Big(\frac{\bigl(\phi_{\tau, \collection}(u)-1\bigr)\big(1-\sum_{j=1}^{m}p_j^{\tau}\big)\, \frac{z^{\tau}}{\tau!}}{\prod_{j=1}^{m}\Phic(u,p_{j}z)}\;+\;1\Big)\\
&=  \Big(\prod_{j=1}^{m}\Phic(u,p_{j}z)\Big) \times \Big(\frac{\bigl(\phi_{\tau, \collection}(u)-1\bigr)\big(1-\sum_{j=1}^{m}p_j^{\tau}\big)\, \frac{z^{\tau}}{\tau!}}{           \Phic(u,z) - \bigl(\phi_{\tau, \collection}(u)-1\bigr)\Big(1-\sum_{j=1}^{m}p_j^{\tau}\Big)\, \frac{z^{\tau}}{\tau!}    }\;+\;1\Big).\notag
\end{align*}
Solving (\ref{phibound}) for $\prod_{j=1}^{m}\Phic(u,p_{j}z)$,
plugging into the denominator of the line above, and simplifying, we obtain
\begin{align*}
\Phic(u,z)&=  \Big(\prod_{j=1}^{m}\Phic(u,p_{j}z)\Big) \times \Big( \frac{1}{\frac{\Phic(u,z)}{ (\phi_{\tau, \collection}(u)-1)(1-\sum_{j=1}^{m}p_j^{\tau})\, \frac{z^{\tau}}{\tau!}} -1}+1\Big).
\end{align*}

To simplify the notation, we define $u_{n} := u/\sqrt{V_{n}}$\thinspace.
Now we want to bound the rightmost term close to~$1$.
To do that, we note that for any $\alpha>0$ we may assume that
$|\Phic(u_{n},z)| \geq e^{|z|^{1-\alpha}}$, because if this is not so, our
induction hypothesis (which appears ahead, at~(\ref{induchyp})) is
already satisfied. 
With this assumption, we obtain
\begin{align*}
\frac{|\Phic(u_{n},z)|}{ \bigl(\phi_{\tau, \collection}(u_{n})-1\bigr)\Big(1-\sum_{j=1}^{m}p_j^{\tau}\Big)\, \frac{|z|^{\tau}}{\tau!}}\geq \frac{e^{|z|^{1-\alpha}}}{\frac{\tau u_{n}}{1-\tau u_{n}} \Big(1-\sum_{j=1}^{m}p_j^{\tau}\Big)\, \frac{|z|^{\tau}}{\tau!}}.
 \end{align*}
From here we compute 
\begin{align*}
\Phic(u_{n},z)=\Big(\prod_{j=1}^{m}\Phic(u_{n},p_{j}z)\Big) \times \bigl(1+ O(n^{-1/2}|z|^{\tau}e^{-|z|^{1-\alpha}})\bigr).
\end{align*}
Taking the logarithm and bounding, we find that
\begin{align}\label{logrec}
\big|L(u_{n},z)\big| \leq \sum_{j=1}^{m} \big|L(u_{n},p_{j}z)\big| + C |z|^{\tau}e^{-|z|^{1-\alpha}}n^{-1/2},
\end{align}
where we write $L(u,z):=\log(\Phic(u,z)).$ Here the constant $C$ depends on $u$ and $p_{j}$, but is independent of $n$. 

We now state our inductive hypothesis: For $C$ as given in~(\ref{logrec}) and some constant $A>0$, we have
\begin{align}\label{induchyp}
|L(u_{n},z)| \leq |z|-A|z|^{1/2+\epsilon}+ C n^{-1/2} \sum_{\ell=0}^{k}\sum_{w \in \mathcal{A}^{\ell}}|\Pw z|^{\tau}e^{-\Pw|z|^{1-\alpha}},
\end{align}
for every $z \in D_{k}\cap \mathcal{C}(\epsilon)^{C}$ such that $|z| \leq n$, where  $\mathcal{C}(\epsilon)^{C}$ denotes the complement of the cone $\mathcal{C}(\epsilon)$.

We first handle the $k=0$ step. In $D_{0} \cap \mathcal{C}(\epsilon)^{C}$ our hypothesis is that 
\begin{align}\label{inithyp}
\big|L(u_{n},z)\big| \leq |z|-An^{-1/2}|z|^{1/2+\epsilon}+C|z|^{\tau}e^{-|z|^{1-\alpha}}.
\end{align}
Now, by choosing our starting-radius $\xi$ and our starting $n$-value large enough, we can guarantee that $|\Phic(u_{n},z)|$ is as close as we like to $|e^{z}|$. And we know that $|e^{z}|<e^{|z|}$ since we are outside a cone $\mathcal{C}(\epsilon)$ which contains the positive real axis. So under these circumstances, we can always find $A$ such that
$$\big|\Phic(u_{n},z)\big|< \exp\bigl(|z|-A|z|^{1/2+\epsilon}+Cn^{-1/2}|z|^{\tau}e^{-|z|^{1-\alpha}}\bigr).$$
Since $L(u_{n},z)=\log(\Phic(u_{n},z))$, we can conclude that~(\ref{inithyp}) holds (though we may have to adjust $\xi$ slightly,
since for every $x$ we have $\log|x| \leq |\log(x)|$ which is the wrong direction
for us here).  This concludes the initial step.

For the inductive step, we assume that~(\ref{induchyp}) holds for some $k$, and take $z \in D_{k+1} \cap \mathcal{C}(\epsilon)^{C}$. Then, by our bounded recurrence~(\ref{logrec}), we have
\begin{align*}
|L(u_{n},z)| &\leq \sum_{j=1}^{m}\bigg(|p_{j}z| - A|p_{j}z|^{1/2+\epsilon}+ Cn^{-1/2}\sum_{\ell=0}^{k}\sum_{w \in \mathcal{A}^{\ell}}(\Pw p_{j}z)^{\tau}e^{-\Pw |p_{j}z|^{1-\alpha}}\bigg)+  Cn^{-1/2}|z|^{\tau}e^{-|z|^{1-\alpha}}\\
& \leq |z| - A|z|^{1/2+\epsilon} +Cn^{-1/2}\sum_{\ell=0}^{k+1}\sum_{w \in \mathcal{A}^{\ell}}(\Pw z)^{\tau}e^{-\Pw |z|^{1-\alpha}},
\end{align*}
which shows that the induction hypothesis holds for $k+1$. We can then conclude that our hypothesis holds on the intersection of the circle $\{|z|=n\}$ and the complement of the cone $\mathcal{C}(\epsilon)$. 

It remains, however, to bound the extraneous term in~(\ref{induchyp}),
which is not found in Condition~3 of  Lemma~\ref{Lm:distdepoi}.
We can obtain the requisite bound by using the Mellin transform. We rephrase the formulation as
\begin{align}\label{xtoy}
\sum_{\ell=0}^{k}\sum_{w \in \mathcal{A}^{\ell}}(\Pw z)^{\tau}e^{-\Pw |z|^{1-\alpha}}=z^{\tau\alpha}\sum_{\ell=0}^{k}\sum_{w \in \mathcal{A}^{\ell}}(\Pw  z^{1-\alpha})^{\tau}e^{-\Pw |z|^{1-\alpha}},
\end{align}
and then define
\begin{align*}
f(y)=\sum_{w \in \mathcal{A}^{*}}(\Pw y)^{\tau}e^{-\Pw y},
\end{align*}
so that $|z^{\tau\alpha}|f(|z|^{1-\alpha})$ bounds the right-hand side of~(\ref{xtoy}).
Taking the Mellin transform of $f(y)$, we obtain
\begin{align*}
\fs(s)=\sum_{w \in \mathcal{A}^{*}}\Pw^{-s}\Gamma(s+\tau)=\frac{\Gamma(s+\tau)}{1-\sum_{j=1}^{m}p_{j}^{-s}},
\end{align*}
which is valid in the strip $\langle -\tau,-1 \rangle$. We then evaluate the inverse Mellin integral by taking the residue at $s=-1$, and closing the box, and recover the value
\begin{align*}
f(y)= \frac{\Gamma(\tau-1)}{h}\;y\bigl(1+\delta(y)+ o(1)\bigr), 
\end{align*}
where $\delta(y)$ is a fluctuating function of bounded magnitude if our probability-family is periodic, and $0$ otherwise. Now, since  $|z^{\tau\alpha}|f(|z|^{1-\alpha})$ bounds the right-hand side of~(\ref{xtoy}), we have
\begin{align*}
\sum_{\ell=0}^{k}\sum_{w \in \mathcal{A}^{\ell}}(\Pw |z|)^{\tau}e^{-\Pw |z|^{1-\alpha}} \leq  
\frac{1}{h}|z|^{1+(\tau-1)\alpha}\;\Gamma(\tau-1)\bigl(1+o(1)\bigr).
\end{align*}
In our inductive hypothesis~(\ref{induchyp}),
the left-hand side of the above inequality is preceded by $n^{-1/2}$; the overall effect is of a term of order $n^{1/2+(\tau-1)\alpha}$. So as long as we choose $\alpha$ satisfying $(\tau-1)\alpha<\epsilon$, this term is subsumed by the other two, and Condition 3 of  Lemma~\ref{Lm:distdepoi} is proven to hold for $Y_{n,\collection}$. And this proves convergence to a normal distribution, as claimed in Theorem~\ref{Theo:main}. 
\section{Examples}
\label{Sec:examples}
In this section we provide some examples of particular collections that appear often
in writings on this subject. 
\subsection{A binary trie}
The trie arising on the alphabet $\{0,1\}$ is common 
in computer applications as 
a data structure. Suppose the probability of 1 is $p$, 
and the probability of 0 is $q = 1 - p$.
Here, the entropy of the probability source is
$h = - p \ln p - q \ln q $. Suppose the motif $T$ 
is the tree in Figure~\ref{Fig:binarymotif}.

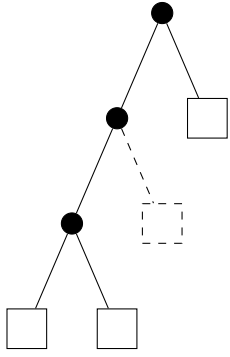
\begin{figure}[!ht]
\begin{center}
\begin{tikzpicture}
  [level distance=14mm,
    level 1/.style={sibling distance=12mm},
    level 2/.style={sibling distance=12mm},
    level 3/.style={sibling distance=12mm},
    ]
\tikzset{every node/.style={minimum size = 8, inner sep=0}}  
\node[circle, draw,  fill = black]{$\phantom{.}$}
      child{node[circle, draw,  fill = black]{\phantom{.}}
        child{node[circle, draw,  fill = black]{}
           child{node[minimum size = 15, inner sep=0, draw]{}}
           child{node[minimum size = 15, inner sep=0, draw]{}}
        }
        child{node[minimum size = 15, inner sep=0, draw, dashed]{}
          edge from parent [dashed]}
      }
      child{node[minimum size = 15, inner sep=0, draw]{}};
\end{tikzpicture} 
\end{center}
\caption{A binary motif.}
\label{Fig:binarymotif}
\end{figure}

This motif occurs in a trie of size 3 with probability 
$6(p^3) (p^2 q) q = 6 p^5 q^2$, as given
in Remark~\ref{Remark:shapefunctional}. In a large trie of size $n$, this motif occurs 
$$\E[X_{n,T}] = \frac {   
p^5 q^2} 
          { h } \, n  + n \xi_T(n) + o(n)$$
times on average.
The function $\xi_T(.)$ is zero, when $\frac {\ln p} {\ln q}$ is irrational. When
$\frac {\ln p} {\ln q}$ is rational, $\xi_T(.)$ is an oscillating function. For example,
in the unbiased case $p = q = \frac 1 2$, where $\frac {\ln p} {\ln q} =1$,
this oscillating function is
$$ \xi_T (n) = {3  \over 512 \, \ln 2}\sum_{j\in Z\setminus\{0\}}
\Gamma\Bigl(2+{2\pi i j\over
 \ln 2}\Bigr)\exp(-2 \pi i j \log_2 n).$$
Uniformly in $n$ this function is bounded by
$0.4568554688\times 10^{-5}$.

For general $p$ and $q$, we have
\begin{align*}
\V[X_{n,T}] := v(n)&={n \over h} {p^5q^2}-{2n \over h} p^{10} q^4 \bigg[\frac 5 {16}
+\frac 1 {36} \sum_{j= 0}^\infty (-1)^j {p^{j+3} + q^{j+3} \over 1 -  p^{j+3} - q^{j+3}} \times {(j+4)! \over j!}\bigg] + n\delta_T(n)+o(n),
\end{align*}
where $\delta_T(.)$ is an 
oscillating function 
(identically zero in the case $\frac {\ln p} {\ln q}$ is
irrational).

The Gaussian law is
$$\frac {X_{n, T} - \Bigl(\displaystyle\frac {   
p^5 q^2} 
          { h }  +  \xi_{T}(n)\Bigr) \, n } {\sqrt {v(n)}} \ \convLaw \ \normal(0, 1).$$
\subsection{Two tries from DNA data}
In the hypervirus genome DNA model, the probabilities of the nucleotides
$A, C, T, G$ are respectively $0.15, 0.35, 0.35, 0.15$, and they are assumed to
be independent. 
This frequency distribution has the approximate entropy 1.304011483. 
The strands of DNA are very long
and the infinite string model provides an approximation. 
Let us use $T$ and $\widetilde{T}$ to denote the
motifs on the left and right (respectively) of Figure~\ref{Fig:motifs}.
The motif $T$ has shape functional $(4!)
(0.15^2) (0.15 \times 0.35) (0.15^2) (0.15) 
= 0.00009568125$,
and the motif $\widetilde T$ has shape functional
$(4!)(0.15)(0.35^2) (0.35^2) (0.15)  
= 0.0081034$.

Ignoring fluctuations,
in the trie of $n$ (very large) random DNA strands, we have
\begin{alignat*}{3}
\E[X_{n, T}] &\approx                0.000006115 \, n, 
  & &\qquad & \E[X_{n, \widetilde T}] &\approx 0.000517849 \, n,\\
\V[X_{n, T}] &\approx 0.000006114 \, n,
  & &\qquad & \V[X_{n, \widetilde T}] &\approx 0.000516520\, n,
\end{alignat*}
and
\begin{equation*}
\Cov\bigl[X_{n, T}, X_{n, \widetilde T}\bigr] = -1.56934066\times 10^{-8}\, n.
\end{equation*}
The distribution of the number of occurrences of these two motifs has 
approximately bivariate normal distribution.
$$\begin{pmatrix} X_{n,T}\\ X_{n,\widetilde T}  \end{pmatrix}
    \ \approxLaw \ \normal_2 \biggl(
        \begin{pmatrix} 0.000006115 \\
        0.000517849
        \end{pmatrix} \, n ,\begin{pmatrix}
        0.000006114 & -1.56934066\times 10^{-8} \\
        -1.56934066\times 10^{-8} & 0.000516520\end{pmatrix}\, n \biggr).$$
\subsection{The number of $\tau$--cousins}
Let $\collection$ be the collection of all $\tau$--cousins (all tries of size $\tau$).
For $\tau = 2$, there is only one 2--cousin ($\mathcal{I} = 1$). 
 (Here we are defining $\collection$ to just be
  $\tau$--cousins on the fringe, so these cherries will not have any
  extraneous strings at the top of the subtree, attached to them.)
However, for $\tau\ge 3$,
there is a countably infinite number of $\tau$--cousins, 
so we can take $\mathcal{I}$ to be the set of natural numbers.
Let $\mathcal{Y}_{n,\tau}$ be the number of $\tau$--cousins, so it is the linear combination
$$\mathcal{Y}_{n, \tau} = \sum_{\nu \in \indexingset} X_{n, T_\nu}.$$ 
According to the calculation of the average of a linear combination, we have  
$$\E[\mathcal{Y}_{n,\tau}] = \frac {\sum_{j=0}^\infty \Q(T_j)} 
          {\tau(\tau-1)h} \, n  + n \xi_\collection^*(n) + o(n) =  \frac {1 - \sum_{j=1}^m p_j^\tau}
          {\tau(\tau-1)h} \, n  + n \xi_\collection^*(n) + o(n),$$
where $\xi_\collection^*(.)$ is an 
oscillating function that collects all the individual oscillations.
We thus recover the result in~\cite{Ward}.
The variance of this linear combination (with all $\alpha$'s being 1) is
\begin{align*}\label{Eq:LinearVariance}
\V[\mathcal{Y}_{n,\collection}]&= \bigg({1 -  \sum_{j=1}^m p_j^\tau\over \tau(\tau-1)}\, -\frac 2 {(\tau!)^2}  \Bigl(1 -  \sum_{j=1}^m p_j^\tau\Bigr)^2\\
& \qquad \qquad {}\times \bigg[ \frac {(2\tau-2)!} {2^{-2\tau}}
+\sum_{j = 0}^\infty (-1)^j {\sum_{k=1}^m p_k^{\tau+j} \over 1 - \sum_{k=1}^m p_k^{\tau+j}} \times {(2\tau+j-2)! \over j!}\bigg]\bigg) \, \frac h n  + \delta_\collection^* (n)n\\
& \qquad {}-\bigg({1 -  \sum_{j=1}^m p_j^\tau \over \tau(\tau-1)h}+\widehat{\delta}_{\collection}^*(n)\bigg)^{2} n+o(n),
\end{align*}
where $\delta_\collection^* (.)$ and $\widehat\delta_\collection^*(.)$ are 
oscillating functions (possibly 0).

The number of $\tau$--cousins follows a Gaussian law:
$$\frac {\mathcal{Y}_{n,\tau} -  \Bigl(\displaystyle \frac {1 - \sum_{j=1}^m p_j^\tau} 
          {\tau(\tau-1)h} \,   +  \xi_\collection^* (n)\Bigr) \, n}   {\displaystyle\Big({1 -  \sum_{j=1}^m p_j^\tau \over \tau(\tau-1)h}+\widehat{\delta}_{\collection}^*(n)\Big)
          \sqrt n} \ \convLaw \ \normal(0,1).$$
          
The different $\tau$--cousins are countable, and can be enumerated appropriately.
We can call them $K_1, K_2, \ldots$, etc.
As a consequence of Theorem~\ref{Theo:main}, the number of cousins $X_{n,K_i}$,
for $i=1, 2,\ldots$, together have an asymptotic multivariate distribution.
For instance for 3--cousins, with a binary alphabet, we can think of $K_{2i-1}$ 
as being the cousin with one right leaf, and a left path of length $i$ then splitting
into two leaves, and take its mirror image as $K_{2i}$. 
With an aperiodic binary alphabet, the multivariate central limit theorem
takes the form
$$\frac {\begin{pmatrix} X_{n, K_1} \\
                         X_{n, K_2}\\
                         \vdots \end{pmatrix} - \begin{pmatrix}   p^3 q^2 \\
                         p^2 q^3\\
                         \vdots \end{pmatrix} \displaystyle \frac n {6h}} 
                             {\sqrt n} \ \convLaw \ \normal_2\left(\begin{pmatrix}   0 \\
                         0\\
                         \vdots \end{pmatrix} ,  
                             \begin{pmatrix}
                                       \sigma_{1,1}^2    &  \sigma_{1,2}^2& \cdots \\
                                       \sigma_{1,2}^2    &  \sigma_{2,2}^2& \cdots\\
                                      \vdots&\vdots&\ddots
                                      \end{pmatrix}\right),$$
where $h = -p \ln p - q \ln q$ is the entropy of the alphabet, and $\sigma_{i,j}$,
$1 \le i, j\le \infty$, are the linearity coefficients in the variances and covariances
given in Theorem~\ref{Theo:covar}.

\section{Acknowledgements}

The authors sincerely thank an anonymous referee for detailed and
insightful comments about the entire paper.  We acknowledge the
referee for improving the paper in many ways.
M.~D. Ward's research is supported by 
NSF Grant DMS-1246818, and by the NSF Science \& Technology Center for
Science of Information Grant CCF-0939370.

\end{document}